\documentclass[12pt]{amsart}
\pdfoutput=1
\usepackage{microtype}
\usepackage{etex} 
\usepackage[active]{srcltx}
\usepackage{calc,amssymb,amsthm,amsmath,amscd, eucal,ulem}
\usepackage{alltt}
\RequirePackage[dvipsnames,usenames]{color}

\input{mabliautoref.sty}
\input{xy}
\xyoption{all}
\input{kmacros3.sty}
\usepackage{tikz}
\usepackage{amsfonts, mathrsfs}
\usepackage{mathtools}
\usepackage{calligra}
\usepackage{stmaryrd} 
\usepackage{dcpic, pictexwd}
\usepackage[left=1in,top=1in,right=1in,bottom=1in]{geometry}
\usepackage{bm}
\usepackage{verbatim}

\usepackage[cal=boondox, calscaled=1.05]{mathalfa}

\numberwithin{equation}{theorem}

\newcommand{\kay}{\mathcal{k}}

\renewcommand{\:}{\colon}
\newcommand{\eg}{{\itshape e.g.} }
\renewcommand{\m}{\mathfrak{m}}

\DeclareMathOperator{\vol}{vol}

\DeclareMathOperator{\Bl}{Bl}

\DeclareMathOperator{\Cl}{Cl}

\DeclareMathOperator{\ssHom}{ \sH\!\!\;\!\!\text{\calligra{\Large om}\,}}

\usepackage{setspace}
\usepackage{hyperref}
\usepackage{url}
\let\oldFootnote\footnote
\newcommand\nextToken\relax

\renewcommand\footnote[1]{%
    \oldFootnote{#1}\futurelet\nextToken\isFootnote}

\newcommand\isFootnote{%
    \ifx\footnote\nextToken\textsuperscript{,}\fi}

\usepackage{enumitem}
\usepackage{graphicx}
\usepackage[all,cmtip]{xy}

\usepackage{verbatim}

\theoremstyle{theorem}

\renewcommand{\O}{\mathcal{O}}
\usepackage{marvosym}

\begin{document}
\title{Varieties with ample Frobenius-trace kernel} 
\author[J.~Carvajal-Rojas]{Javier Carvajal-Rojas}
\address{Centro de Investigaci\'on en Matem\'aticas, A.C., Callej\'on Jalisco s/n, 36023 Col. Valenciana, Guanajuato, Gto, M\'exico}
\email{\href{mailto:javier.carvajal@cimat.mx}{javier.carvajal@cimat.mx}}
\author[Zs.~Patakfalvi]{Zsolt Patakfalvi}
\address{\'Ecole Polytechnique F\'ed\'erale de Lausanne\\ SB MATH CAG\\MA C3 635 (B\^atiment MA)\\ Station 8 \\CH-1015 Lausanne\\Switzerland}
\email{\href{mailto:zsolt.patakfalvi@epfl.ch}{zsolt.patakfalvi@epfl.ch}}
\keywords{Cokernel of Frobenius, Cartier operators, Frobenius traces, Kunz's theorem, Mori--Hartshorne's theorem.}

\thanks{The authors were partially supported by the ERC-STG \#804334. Carvajal-Rojas was partially supported by the grants FWO \#G079218N, CONAHCYT \#CBF2023-2024-224, and CONAHCYT \#CF-2023-G-33.}

\subjclass[]{14G17, 14J45, 14E30, 14J30, 14J26}


\begin{abstract}
In the search for a projective analog of Kunz's theorem and a Frobenius-theoretic analog of  Mori--Hartshorne's theorem, we investigate the positivity of the kernel of the Frobenius trace (equivalently, the negativity of the cokernel of the Frobenius endomorphism) on a smooth projective variety over an algebraically closed field of positive characteristic. For instance, such a kernel is ample for projective spaces. Conversely, we show that for curves, surfaces, and threefolds, the Frobenius trace kernel is ample only for Fano varieties of Picard rank $1$.
\end{abstract}

\maketitle


\section{Introduction}

A theorem is missing in algebraic geometry. Namely, we are missing a projective analog of Kunz's theorem characterizing regularity \cite{KunzCharacterizationsOfRegularLocalRings} and at the same time a Frobenius-theoretic analog of Mori--Hartshorne's characterization of projective space \cite{MoriCharacterizationProjSpace,HartshorneAmpleSubvarieties} (originally known as Hartshorne's conjecture). We will elaborate on this next, for which we fix an algebraically closed field $\kay$ of characteristic $p \geq 0$.

Let us consider two important dichotomies in algebraic geometry: local vs. global and characteristic zero vs. positive characteristic. For instance, the local study of coherent sheaves surrounds the notion of freeness/flatness, whereas globally it focuses on positivity (e.g. ampleness). Likewise, characteristic zero geometry is governed by differentials $\Omega^1$, whereas on positive characteristic geometry the Frobenius endomorphisms must be taken into account.  Thus, with respect to the above two dichotomies, there are four scenarios in which one may do algebraic geometry. We claim that there is a theorem on three of these scenarios and an analogy between them, but the analogous theorem is missing in the fourth scenario. The situation is summarized as follows:

\begin{center}
\begin{tabular}{|c|c|c|}
\hline
                      & Local (singularities) & Global (projective geometry)   \\ \hline
Differentials & Jacobian criterion    & Mori--Hartshorne's theorem       \\ \hline
Frobenius ($p>0$)     & Kunz's theorem        & \textbf{?}                               \\ \hline
\end{tabular}
\end{center}

For the reader's convenience, we briefly recall these three prominent theorems. Let us start with Kunz's theorem \cite{KunzCharacterizationsOfRegularLocalRings} and assume that $p>0$. Kunz's theorem establishes that a variety $X/\kay$ is smooth if and only if $F_* \sO_X$ is locally free of rank (necessarily) $p^{\dim X}$, where $F=F_X\: X \to X$ denotes the (absolute) Frobenius endomorphism of $X$. Equivalently, let us consider the exact sequence
\[
0 \to \sO_X \xrightarrow{F^{\#}} F_* \sO_X \to \sB^1_X \to 0
\]
defining $\sB^1_X$ as the \emph{cokernel of Frobenius}. Then, Kunz's theorem can be rephrased by saying that $X/\kay$ is smooth if and only if $\sB_X^1$ is locally free of rank (necessarily) $p^{\dim X}-1$.\footnote{Technically speaking, Kunz's theorem characterizes the regularity of $X$ rather than the smoothness of $X/\kay$.} 

Compare this to the jacobian criterion: a variety $X/\kay$ is smooth if and only if $\Omega^1_{X/\kay}$ is locally free of rank $\dim X$. Thus, the smoothness of a variety $X/\kay$ can be determined using either $\Omega_{X/\kay}^1$ or $\sB_X^1$.

There is a purely local way to look at Kunz's theorem. Set $\hat{\bA}^d_{\kay} \coloneqq \Spec \kay \llbracket x_1, \ldots, x_d \rrbracket$. By the Cohen structure theorem, the spectra of noetherian complete local $\kay$-algebras $(A,\fram,\kay)$ are, up to isomorphism, the closed subschemes of $\hat{\bA}^n_{\kay}$ for some $n$, and $\hat{\bA}^d_{\kay}$ is (up to isomorphism) the only regular one of dimension $d$. We refer to these spectra simply as $\kay$-singularities. Thus, Kunz's theorem establishes that $\hat{\bA}^d_{\kay}$ is characterized among $d$-dimensional $\kay$-singularities as the one and only one whose Frobenius has a free cokernel. Of course, an analogous characterization can be obtained using differentials by the jacobian criterion.

On the other hand, we may consider Mori's theorem (originally called Hartshorne's conjecture) characterizing the projective spaces among smooth projective varieties by the ampleness of the tangent sheaf \cite{MoriCharacterizationProjSpace}, \cf \cite{HartshorneAmpleSubvarieties,MabuchiC3ActionsOnAlgebraicThreefoldsWithAmpleTangentBundle,MoriSumihiroOnHartshorneConjecture}. Concretely, a $d$-dimensional smooth projective variety $X/\kay$ has an ample tangent sheaf $\sT_{X/\kay} \coloneqq \Omega_{X/\kay}^{1,\vee}$ if and only if $X \cong \bP^d_{\kay} \coloneqq \Proj \kay[x_0,\ldots,x_d]$. We refer to this as the Mori--Hartshorne theorem. See \cite[V, Corollary 3.3]{KollarRationalCurvesOnAlgebraicVarieties} for a treatment in an arbitrary (equal) characteristic.

Further, a direct graded-algebra computation shows that
\begin{equation} \label{DefAie}
F_* \sO_{\P^d_{\kay}} \cong \sO_{\P^d_{\kay}} \oplus \sO_{\P^d_{\kay}}(-1)^{\oplus a_{1}} \oplus \cdots \oplus \sO_{\P^d_{\kay}}(-d)^{\oplus a_{d}},
\end{equation}
where the integers $a_{1},\ldots,a_{d}$ are uniquely determined by such isomorphism. Moreover,
\[
 \sB^{1,\vee}_{\bP^d_{\kay}} \cong \sO_{\P^d_{\kay}}(1)^{\oplus a_{1}} \oplus \cdots \oplus \sO_{\P^d_{\kay}}(d)^{\oplus a_{d}}
\]
is ample.

In view of all the above, it is inevitable to wonder:

\begin{question} \label{que.FirstQuestionAmplenessE}
 Let $X/\kay$ be a $d$-dimensional smooth projective variety such that the locally free sheaf $\sE_X \coloneqq \sB^{1,\vee}_X$ is ample, which experts will quickly recognize as the kernel of the Frobenius trace $\tau_X \: F_* \omega_X^{1-p} \to \sO_X$ (see \autoref{sec.FrobeniusCartierOperators}). Is $X$ isomorphic to $\bP^d_{\kay}$?
\end{question}

If \autoref{que.FirstQuestionAmplenessE} were to have an affirmative answer, we may think of it as both a projective Kunz's theorem and a Frobenius-theoretic Mori--Hartshorne's criterion. That is, we could tell the projective spaces apart among smooth projective varieties by the ampleness of a locally free sheaf naturally defined via its Frobenius. Unfortunately, \autoref{que.FirstQuestionAmplenessE} has a negative answer for all $d \geq 3$. Indeed, using the description in \cite{LangerDAffinityFrobeniusQuadrics,AchingerFrobeniusPushForwardQuadrics}, $\sE_X$ can be seen to be ample already for quadrics of dimension $d \geq 3$ and $p \geq 3$; see \autoref{ex.QuadricsAnalysis}. 

On the positive side, we show that if $\sE_X$ is ample then $X$ is a Fano variety. In general, if $\sE_X$ has a certain positivity property, then the same property holds for $\omega_X^{-1}$; see \autoref{thm.AmpleImpliesFano}. In particular, \autoref{que.FirstQuestionAmplenessE} has an affirmative answer for $d=1$. We are able to verify this for surfaces as well. For threefolds, we managed to reduce the class of Fano threefolds for which $\sE_X$ is ample via an extremal contraction analysis. We obtain the following result.

\begin{mainthm*}[{\autoref{thm.MainTheorem}, \autoref{cor.MainCorDim2}, \autoref{cor.MainCorDim1}}]
Let $X/\kay$ be a $d$-dimensional smooth projective variety such that $\sE_X = \ker(\tau_X \: F_* \omega_X^{1-p} \to \sO_X)$ is ample and $d\leq 3$. Then $X$ is a Fano variety of Picard rank $1$. 
\end{mainthm*}
However, its converse seems rather tricky. Except for the projective space and the quadric, we do not know whether $\sE_X$ is ample for Fano threefolds of Picard rank $1$ (i.e. for those of index $1$ or $2$). However, if $X$ is the quadric, $\sE_X$ is ample if and only if $p \neq 2$ (see \autoref{cor.Quadrics}), which we find rather baffling. We leave this converse problem open:

\begin{question} \label{que.ProjectieKunztheoremWeakForm}
For which Fano threefolds of Picard rank $1$ is $\sE_X$ ample? Does this depend on the characteristic as it does for quadrics? 
\end{question}

We may still wonder whether there is a locally free sheaf naturally defined via Frobenius that can be used to tell the projective space apart among Fano varieties of Picard rank $1$. Our failed, first attempt was to use the dual of the cokernel of Frobenius $\sB^1_X$, which is none other than the kernel of the \emph{first} Cartier operator on $X$; see \autoref{rem.GeneralCartierOperators}. Nonetheless, there are $d+1$ Cartier operators $\{\kappa_i \: F_* Z^i \to \Omega_{X/
\kay}^i\}_{i=0}^d$ attached to a $d$-dimensional smooth projective variety $X/\kay$, where $Z^i \subset \Omega^i_{X/\kay}$ is the subsheaf of exact forms. For instance, the $d$-th Cartier operator is the usual $F_* \omega_X \to \omega_X$ giving $\omega_X$ its natural Cartier module structure. Letting $\sB^i_X$ be the kernel of the $i$-th Cartier operator, we may wonder whether the ampleness of $\sB_X^{\bullet, \vee} \coloneqq \bigoplus_{i=0}^d\sB_X^{i,\vee}$ may be used to characterize the projective space among smooth projective varieties.\footnote{It turns out that $\sB_X^{0,\vee}=0$ and so we may ignore the $0$-th direct summand.} On the other hand, the following relation is well-known
\[
\sB_X^{d,\vee} \cong \sB^1_X \otimes \omega_X^{-1} \cong (\sE_X \otimes \omega_X)^{\vee};
\]
see \autoref{rem.GeneralCartierOperators}. From this, we see that $\sB_X^{d,\vee}$ is ample for all projective spaces and all quadrics of dimension $\geq 3$. However, 
\begin{question}
Can we use the ampleness of $\sB^{2,\vee}$ to distinguish between $\bP^3_{\kay}$ and the threefold quadric $\bQ^3/\kay$?
\end{question}
Unfortunately, we do not know whether $\sB^{2,\vee}_{\bP^3}$ is ample, and also whether $\sB^{2,\vee}_{\bQ^3}$ is not ample.

Roughly speaking, the study of projective varieties can be reduced to the study of graded commutative algebra. Thus, one may expect that one could obtain a \emph{reasonable} projective Kunz theorem by translating the corresponding graded Kunz's theorem. For the reader's convenience, we have worked out this easy translation in \autoref{sec.GradedKUNZ}. See \autoref{cor.GradedKunzTheorem} for the statement, which we find rather unsatisfactory as it does not resemble Mori--Hartshorne's theorem. Nevertheless, it at least indicates that the structure of Frobenius pushforwards can be used to characterize projective spaces among smooth projective varieties.

Last but not least, let us highlight an interesting side application of our methods. Let $(V,\sA)$ be a polarized projective normal variety (i.e. $\sA$ is an ample invertible sheaf in $V$) with a corresponding affine cone $X$ and a vertex point $0 \in X$. In \autoref{sec.CONESEXAMPLES}, we outline a geometric method for describing explicitly $F_* \sO_{X,0}$ as an $\sO_{X,0}$-module. The only input needed is an explicit description of $F_* \sA^i$ for $i=0,\ldots, p-1$. We illustrate this method by carrying out the computations for both Veronese and Segre embeddings. In principle, one can do this for quadric cone singularities as well by using \cite{AchingerFrobeniusPushForwardQuadrics,LangerDAffinityFrobeniusQuadrics} as input. However, we will attempt this rather lengthy calculation elsewhere. The authors are unaware of any such explicit descriptions in the literature and believe that this could be of interest to commutative algebraists. See \autoref{rem.ApplicationLocalAlgebraConesRationalNormalCurve}, \autoref{rem.ApplicationLocalAlgebraVeronese}, \autoref{rem.ApplicationLocalAlgebraSegreProduc}.

\subsection*{Outline} This paper is organized as follows. \autoref{sec.FrobeniusCartierOperators} briefly surveys the basics on Cartier operators (i.e. $\kappa_X \: F_* \omega_X \to \omega_X$) and Frobenius traces (i.e. $\tau_X = \kappa_X \otimes \omega_X^{-1} \: F_* \omega_X^{1-p} \to \sO_X$) and so it may be skipped by experts. In \autoref{sec.FrobpushforwardsProjectiveBundles}, we compute directly the Frobenius pushforwards of invertible sheaves on projective bundles of the form $\bP(\sL_0 \oplus \cdots \oplus \sL_d) \to X$; see \autoref{pro.PushforwardProjectiveBundles}, which we later use in \autoref{sec.Examples} to calculate several examples of interest (e.g. blowups along linear subspaces of projective spaces). Most importantly, this is used to prove that $\sE_X$ is never ample if $X$ is a blowup of a smooth variety along a smooth subvariety; see \autoref{lem.FrobPushForwardLocalgeneralBlowups}. Finally, in \autoref{sec.MainSection}, we study the repercussions the ampleness of $\sE_X$ (as well as other positivity conditions) has on $X$. Our first basic observation is that there is a surjective morphism $F^* \sE_X \to \omega_X^{1-p}$, which allows us to conclude that $\omega_X^{-1}$ is ample if so is $\sE_X$; see \autoref{thm.AmpleImpliesFano}. We further narrow this down to \autoref{Main.SubsectionFibrationsContractions}, where we investigate the interplay between the ampleness of $\sE_X$ and extremal contractions of smooth threefolds. In a nutshell, if $X$ is a Fano variety of dimension $\leq 3$ and $\sE_X$ is ample, then it admits no extremal contraction except for $X \to \Spec \kay$ (where the whole space is contracted to a point) and so $X$ is a Fano variety of Picard rank $1$.

\begin{convention}
The following conventions are used throughout this paper. We fix an algebraically closed field $\kay$ of characteristic $p>0$. All relative objects (such as varieties) and properties (such as smoothness and projectivity) are defined over $\kay$ unless otherwise explicitly stated. For instance, we write $\bP^d = \bP^d_{\kay}$, $\bA^d = \bA^d_{\kay}$, and so on. We let $F^e = F^e_X \: X \to X$ denote the $e$-th iterate of the absolute Frobenius morphism on a variety $X$. We use the shorthand notation $q\coloneqq p^e$. Given $n \in\bZ$, we use the euclidean algorithm to define
\[
n \eqqcolon \lfloor n/q \rfloor q +[n]_q, \quad 0 \leq [n]_q \leq q-1.
\]
We may drop the subscript from $[n]_q$ if no confusion is likely to occur. If $A$ is a finite set, we denote its cardinality by $|A|$. When no confusion is likely to occur, we may drop subscripts in writing, e.g., $\sO=\sO_X$, $\omega=\omega_X=\omega_{X/\kay}$, $\Omega^i=\Omega^i_{X/\kay}$, etc. Finally, $0 \in \bN$.
\end{convention}

\subsection*{Acknowledgements} 
The authors thank Fabio Bernasconi and Takumi Murayama for very useful discussions. The authors are grateful to their master's student Maxime Matthey for pointing out several typos in previous versions. The authors are also thankful to Pieter Belmans and Michel Brion for providing some interesting references regarding Frobenius pushforwards on certain homogeneous spaces. Last but not least, they thank the anonymous referees for their careful reviews and feedback. 

\section{Generalities on Cartier Operators and Frobenius Traces} \label{sec.FrobeniusCartierOperators}
Throughout this section, we let $X$ be a smooth variety of dimension $d$ and $0 \neq e \in \bN$ be a positive integer. Kunz's theorem establishes that $F^e_* \sL$ is a locally free sheaf of rank $q^d$ for all invertible sheaves $\sL$ on $X$. Since $F^e \:X \to X$ is finite, Grothendieck duality establishes a canonical isomorphism of $F^e_* \sO_X$-modules $F^e_* \sO_X \to \ssHom_X(F^e_* \omega_X , \omega_X)$, whose corresponding global section $\kappa^e =\kappa^e_X  : F^e_* \omega_X \to \omega_X$ is the so-called \emph{Cartier operator} on $X$; see \autoref{rem.GeneralCartierOperators} below. Twisting $\kappa^e$ by $\omega_X^{-1}$ and using the projection formula, we obtain a map $\tau^e = \tau^e_X \:F^e_* \omega_X^{1-q} \to \sO_X$, which we refer to as the \emph{Frobenius trace} on $X$.\footnote{In general, $\sL \otimes F_*^e \sF = F_*^e (\sL^q \otimes \sF)$ for all invertible sheaves $\sL$ and all $\sO_X$-modules $\sF$. This follows from the projection formula and the equality $F^{e,*} \sL = \sL^q$.} Notice that $\tau^e$ is surjective. Indeed, this can be checked locally at stalks where it is clear as $X$ is regular (and so $F$-injective). Thus, there is an exact sequence
\begin{equation} \label{eqn.SESDefiningE_e}
0 \to \sE_{e,X} \to F^e_* \omega_X^{1-q} \xrightarrow{\tau^e} \sO_X \to 0
\end{equation}
defining \emph{$\sE_{e,X}$ as the kernel of the Frobenius trace $\tau_X^e$}. Equivalently,
\begin{equation} \label{eqn.SESAjointOfE_e}
0 \to \sE_{e,X} \otimes \omega_X \to F^e_* \omega_X \xrightarrow{\kappa^e} \omega_X \to 0.
\end{equation}
We may also write $\sE_{e,X} = \sE_{e}$ if no confusion is likely to occur. Observe that $\sE_e$ is a locally free sheaf of rank $q^d-1$ as these short exact sequences are both locally split. From \autoref{eqn.SESAjointOfE_e}, it follows that $\chi(X,\sE_e \otimes \omega_X)=0$. By taking duals, we obtain a short exact sequence
\begin{equation} \label{eqn.FrobeniusSES}
0 \to \sO_X \xrightarrow{(\tau^e)^{\vee}=F^{e,\#}} F^e_* \O_X \xrightarrow{} \sE_{e,X}^{\vee} \to 0.
\end{equation}
In particular, $\chi\bigl(X,\sE_{e,X}^{\vee}\bigr)=0$. Additionally, \autoref{eqn.SESDefiningE_e} splits if and only if so does $F^{e,\#} \: \sO_X \to F^e_* \sO_X$, i.e., $X$ is $F$-split. In that case, $H^i(X,\sE_e \otimes \omega_X)=0$ for all $i$ and so $H^i\bigl(X,\sE_e^{\vee}\bigr)=0$ for all $i$ by Serre duality. Those vanishings are equivalent to Frobenius acting injectively (i.e. semi-simply) on the cohomology groups $H^i(X,\sO_X)$. We shall recall in \autoref{rem.Ordinarity} below that this is the case when $X$ is ordinary.

\begin{remark}[Local description of $\kappa^e \: F_*^e \omega_X \to \omega_X$] \label{rem.LocalDescritptionTraces}
Let $x \in X$ be a closed point. Then, the stalk of $\kappa^e \: F_*^e \omega_X \to \omega_X$ at $x\in X$ is a \emph{Frobenius trace}  $\kappa^e_x \: F_*^e \sO_{x,X} \to \sO_{x,X}$ associated to the regular (and so Gorenstein) local ring $\sO_{X,x}$. To be precise, let $\fram_x = (t_1,\ldots,t_d)$ be a regular system of parameters so that 
\[
 \hat{\sO}_{X,x} \otimes_{\sO_{X,x}} F^e_* \sO_{X,x} = F_*^e \hat{\sO}_{X,x} = \bigoplus_{0 \leq i_1, \ldots, i_d \leq q-1} \hat{\sO}_{X,x} F^e_* t_1^{i_1} \cdots t_d^{i_d}
\]
Moreover, $\hat{\sO}_{X,x} \otimes \Omega_{X}^1 = \bigoplus_{i=1}^d \hat{\sO}_{X,x} \mathrm{d} t_i$ and $\hat{\sO}_{X,x} \otimes \omega_{X} = \hat{\sO}_{X,x} \mathrm{d}t_1 \wedge \cdots \wedge \mathrm{d}t_d$; see \cite{TycDifferentialBasisAndp-basis}.
Let $\Phi^e \: F^e_* \sO_{X,x} \to \sO_{X,x}$ be the projection onto the summand generated by $F^e_* t_1^{q-1} \cdots t_d^{q-1}$ in the above direct sum decomposition, then
\[
\hat{\sO}_{X,x} \otimes \kappa^e \: F_*^e a \mathrm{d}t_1 \wedge \cdots \wedge \mathrm{d}t_d \mapsto \Phi^e(F_*^e a ) \mathrm{d}t_1 \wedge \cdots \wedge \mathrm{d}t_d;
\]
see \cite[Lemma 1.3.6]{BrionKumarFrobeniusSplitting}. In particular, $\kappa_x^e (F^e_* 1) = 0$ as $\Phi^e (F^e_* 1) = 0$.
\end{remark}

\begin{remark}[Generalized Cartier operators] \label{rem.GeneralCartierOperators} We briefly summarize the theory of Cartier operators. For details see \cite[\S 1.3]{BrionKumarFrobeniusSplitting}. Let $\Omega^{\bullet} \coloneqq \Omega_{X/\kay}^{\bullet}$ be the exterior algebra of $X/\kay$, which is a graded-commutative $\sO_X$-algebra. Let $(\Omega^{\bullet},\mathrm{d})$ be the corresponding de~Rham complex. Although this is not a complex of $\sO_X$-modules (as the differentials are not $\sO_X$-linear), $\bigl(F^e_* \Omega^{\bullet},F^e_* \mathrm{d}\bigr)$ is an $\sO_X$-linear complex. Let $Z^{\bullet} \coloneqq Z^{\bullet}_{X/\kay} \subset \Omega^{\bullet}$ be the corresponding graded subspaces of exact forms and $B^{\bullet} \coloneqq B^{\bullet}_{X/\kay} \subset \Omega^{\bullet}$ be the ones of closed forms. Notice that $F^e_* Z^{\bullet}$ is a graded-commutative $\sO_X$-subalgebra of $F^e_*\Omega^{\bullet}$ and $F^e_* B^{\bullet}$ is a graded ideal of $F^e_*Z^{\bullet}$. Thus,
\[
\mathbf{H}^{\bullet}\bigl(F^e_* \Omega^{\bullet} \bigr)= F^e_*Z^{\bullet}/F^e_*B^{\bullet} = F^e_*\bigl(Z^{\bullet}/B^{\bullet}\bigr) = F_*^e \mathbf{H}^{\bullet} (\Omega^{\bullet})
\]
is a graded-commutative $\sO_X$-algebra. The importance of this is that there is a natural isomorphism of graded-commutative $\sO_X$-algebras
\[
C^{-1} = \bigl(C^{-1}\bigr)^{\bullet} \: \Omega^{\bullet} \to \mathbf{H}^{\bullet}\bigl(F^e_* \Omega^{\bullet} \bigr).
\]
The inverse isomorphism $C=C^{\bullet}$ is (are) referred to as the Cartier operator(s) in the literature. For further details, see \cite{BrionKumarFrobeniusSplitting,EsnaultViehwegLecturesOnVanishing,CartierUneNouvelleOperation}. In fact, the composition
\[
F^e_* \omega_X = F^e_* Z^d \to \mathbf{H}^{d}\bigl(F^e_* \Omega^{\bullet} \bigr) \xrightarrow{C^d} \Omega^d = \omega_X
\]
coincides with the Cartier operator $\kappa^e \: F^e_* \omega_X \to \omega_X$ defined via Grothendieck duality. Thus, there is some abuse of terminology. To avoid confusion, we denote the composition \[
F^e_* Z^{\bullet} \to \mathbf{H}^{\bullet}\bigl(F^e_* \Omega^{\bullet} \bigr) \xrightarrow{C^{\bullet}} \Omega^{\bullet}\] 
by $\kappa^e_{\bullet}$. In particular, there is a canonical isomorphism of $\sO_X$-modules $\sE_e \otimes \omega_X \cong F^e_* B^d$, as both are kernels of the same $q^{-1}$-linear map. On the other hand, we have the following exact sequence of $\sO_X$-modules
\[
0 \to \sO_X \xrightarrow{F^{e,\#}} F^e_* \sO_X \xrightarrow{F^e_*\mathrm{d} } F^e_* Z^1 \xrightarrow{\kappa_1^e} \Omega^1 \to 0.
\]
Therefore, there is a canonical isomorphism of $\sO_X$-modules $\sE_e^{\vee} \cong F^e_* B^1$ and a short exact sequence of $\sO_X$-modules
\[
0 \to \sE_e^{\vee} \to F^e_* Z^1 \to \Omega^1 \to 0.
\] 
For ease of notation, we write $\sB_e^i \coloneqq F^e_* B^i$ and $\sZ_e^i \coloneqq F^e_* Z^i$, so that there are exact sequences
\[
0 \to \sB_e^i \to \sZ_e^i \xrightarrow{\kappa_i^e} \Omega^i \to 0,
\]
and further
\[
0 \to \sZ_e^i \to F^e_* \Omega^i \xrightarrow{F^e_* \mathrm{d}} \sB_{e}^{i+1} \to 0.
\]
Summing up, $\sB_e^1$ and $\sB^d_e$ are $\omega$-dual to each other. Further, $(\sB_e^1)^{\vee} \cong \sE_e$ and $(\sB_e^d)^{\vee} \cong (\sE_e \otimes \omega_X)^{\vee} \cong \sE_e^{\vee} \otimes \omega_X^{-1}$.
\end{remark}

\begin{remark}[On the cohomology of $\sE_e^{\vee}$ and ordinarity] \label{rem.Ordinarity}
From the exact sequence
\[
0 \to \bm{\alpha}_q \to \bG_{\mathrm{a}} \xrightarrow{F^e} \bG_{\mathrm{a}} \to 0
\]
on $X_{\mathrm{fl}}$, we obtain canonical isomorphisms $H^i\bigl(X,\sE_e^{\vee}\bigr) \cong H^{i+1}(X_{\mathrm{fl}},\bm{\alpha}_q)$, for $H^i(X_{\mathrm{fl}},\bG_{\mathrm{a}})=H^i(X,\sO_X)$. 
In particular, 
\[
H^0\bigl(X,\sE_e^{\vee}\bigr) \cong \bigl\{ \omega \in H^0\bigl(X,\Omega^1\bigr)  \bigm| \mathrm{d}\omega = 0 \textnormal{ and } C^1 \omega =0  \bigr\} = H^1(X_{\mathrm{fl}},\bm{\alpha}_q).
\]
See \cite[III, Proposition 4.14]{MilneEtaleCohomology}. Following \cite[\S7]{BlochKatopadicetalecohomology}, we say that $X$ is \emph{ordinary} if $H^i(X,B^j)=0$ for all $i,j$. Since $\sB^1_e = F_*^e B^1 = \sE_e^{\vee}$, we have $H^{i}(X,\sE_e^{\vee})=0$ for ordinary varieties, and so the action of Frobenius is injective on $H^i(X,\sO_X)$. Of course, these two notions are equivalent for curves as well as for surfaces as $\sB_e^2=F^e_* B^2 = \sE_e \otimes \omega_X$ by Serre duality.
\end{remark}

The following result will be essential later on.

\begin{proposition}[Naturality of Frobenius trace kernels] \label{rem.NatCarOp}
Let $f\: X \to S$ be a smooth proper morphism between smooth varieties and $e \in \bN$ be a positive integer. Then, there is a canonical commutative diagram between exact sequences
\[
\xymatrix@C=3em{
0 \ar[r] & \sE_{e,X} \ar[r]\ar[d]^-{\varepsilon_{e,X/S}} & F^e_* \omega_X^{1-q} \ar[r]^-{\tau_X^e} \ar[d]^-{\epsilon_{e,X/S}} & \sO_X \ar[r] \ar[d]^-{\cong} & 0 \\
0 \ar[r] & f^*\sE_{e,S} \ar[r] & f^*F^e_* \omega_S^{1-q} \ar[r]^-{f^*\tau_S^e} & f^*\sO_S \ar[r] & 0
}
\]
where $\epsilon_{e,X/S}$ and so $\varepsilon_{e,X/S}$ are surjective. 
\end{proposition}
\begin{proof}
We explain first how $\epsilon_{e,X/S}$ is defined.  Consider the following cartesian diagram
\[
\xymatrix{
X \ar@/_/[ddr]_-{f} \ar@/^/[drr]^-{F^e} \ar@{.>}[dr]|-{F^e_{X/S}}\\
&X^{(q)}  \ar[r]^-{G^e} \ar[d]_-{g} & X\ar[d]^-{f} \\
&S \ar[r]^-{F^e} &S}
\]
defining $F^e_{X/S} \: X \to X^{(q)}$ as the $e$-th relative Frobenius morphism of $f$. Note that $f$ and $g$ are smooth and proper whereas $F^e$, $F_{X/S}^e$, and $G^e$ are finite. Since $f$ and $g$ are smooth and proper, they define the exceptional inverse image functors given by: $f^! = \omega_{X/S} \otimes f^*$ and $g^! = \omega_{X^{(q)}/S} \otimes g^* = G^{e,*} \omega_{X/S} \otimes g^*$, where this last equality follows from
\[
\omega_{X^{(q)}/S} = \det \Omega_{X^{(q)}/S} = \det G^{e,*}\Omega_{X/S} =  G^{e,*} \det \Omega_{X/S} = G^{e,*} \omega_{X/S}.
\]
Thus, $\omega_X = f^! \omega_S $ and $\omega_{X^{(q)}} = g^! \omega_S = G^{e,*}\omega_{X/S} \otimes g^* \omega_S$. The projection formula then yields
\begin{equation} \label{eqn.Idenification}
G^e_* \omega_{X^{(q)}} \cong \omega_{X/S} \otimes G^e_* g^* \omega_S \cong \omega_{X/S} \otimes f^* F_*^e \omega_S = f^! F_*^e \omega_S,
\end{equation}
where the natural transformation $f^* F_*^e \to G^e_* g^*$ is an isomorphism as $f$ is flat. In addition, Grothendieck trace $\gamma^e \: G^e_* \omega_{X^{(q)}} \to \omega_X$ associated to $G^e$ is going to be given by $f^! \kappa_S^e \: f^! F^e_* \omega_S \to f^! \omega_S = \omega_X$ under the natural identification \autoref{eqn.Idenification}. 

By the naturality of Grothendieck trace maps, $\kappa_X^e \: F^e_* \omega_X \to \omega_X$ is the composition of the corresponding traces of $F^e_{X/S}$ and $G^e$. More precisely, if $\kappa_{X/S}^e \: F^e_{X/S, *} \omega_X \to \omega_{X^{(q)}}$ is the relative Cartier operator, then $\kappa_X^e = \gamma^e \circ G^e_* \kappa_{X/S}^e$. 

Putting these observations together, we obtain a canonical factorization
\[
\xymatrix{
F^e_* \omega_X \ar[r]^-{} \ar[rd]_-{\kappa_X^e} & f^! F^e_* \omega_S \ar[d]^-{f^! \kappa^e_S} \\
& \omega_X
}
\]
where the horizontal arrow corresponds to $G_*^e \kappa^e_{X/S}$ under the isomorphism \autoref{eqn.Idenification}. Twisting this diagram by $\omega_X^{-1}$ yields:
\[
\xymatrix@C=3em{
F^e_* \omega_X^{1-q} \ar[r]^-{\epsilon_{e,X/S}} \ar[rd]_-{\tau_X^e} & f^* F^e_* \omega_S^{1-q} \ar[d]^-{f^* \tau^e_S} \\
& \sO_X
}
\]
where the horizontal morphism is the map we aimed to define. 

Finally, we explain why $\epsilon_{e,X/S}\: F^e_* \omega_X^{1-q} \to f^* F^e_* \omega_S^{1-q}$ and so $\varepsilon_{e,X/S} \: \sE_{e,X} \to f^* \sE_{e,S}$ are surjective. It suffices to show that $G^e_* \kappa_{X/S}^e$ is surjective and; since $G^e$ is affine, that $\kappa^e_{X/S}$ is surjective, which can be checked along the geometric fibers of $g$. Note that the pullback of $\kappa_{X/S}^e$ along a geometric point $\bar{s} \to S$ is the (absolute) Cartier operator of the geometric fiber $X_{\bar{s}}$, i.e., $(\kappa_{X/S}^e)_{\bar{s}} = \kappa^e_{X_{\bar{s}}}$ \cite[Lemma 2.16]{PatakfalviSchwedeZhangSingularitiesFamilies}. Since $f$ is smooth, $X_{\bar{s}}$ is regular, and so $\kappa^e_{X_{\bar{s}}}$ is surjective.  
\end{proof}

\section{Frobenius Pushforwards of Invertible Sheaves on Split Projective Bundles} \label{sec.FrobpushforwardsProjectiveBundles}

The goal of this section is to compute explicitly the Frobenius pushforward of an invertible sheaf on a (split) projective bundle (e.g. Frobenius pushforwards of invertible sheaves on projective spaces). From this, we will conclude that $\sE_e$ is ample for projective spaces. We will apply these calculations in \autoref{sec.Examples} to compute $\sE_e$ for several other projective varieties (including some mildly singular ones). Some of those computations will be crucial in our main theorems shown in \autoref{sec.MainSection}.

Let $X$ be a smooth variety and $\sF$ be a locally free sheaf of rank $d+1$ on $X$. Also fix $0 \neq e \in \bN$. We denote by $\varpi \: \bV(\sF) \to X$ and $\pi \: \bP(\sF) \to X$ the respective vector and projective bundles. To be clear on what convention we follow, we have: 
\[
\varpi_* \sO_{\bV(\sF)} = \Sym (\sF) = \bigoplus_{i \in \bZ} \pi_* \sO_{\bP(\sF)}(i).
\]
When no confusion is likely to occur, we may drop $\sF$ from our notation. We set $\sS \coloneqq  \Sym \sF$, $\sS_i = \Sym^i \sF$, and $\sS_+ = \bigoplus_{i \geq 1} \sS_i$. We say that $\varpi$ and $\pi$ are \emph{split} if $\sF$ decomposes as a direct sum of invertible sheaves. In this case, say $\sF = \sL_0 \oplus \cdots \oplus \sL_d$, we write
\[
\sS_i = \bigoplus_{i_0 + \cdots +i_d=i} \sL_0^{i_0} \otimes \cdots \otimes \sL_d^{i_d}
\]

Recall that $\varpi^* \: \Pic X \to \Pic \bV(\sF)$ is an isomorphism. Then, we have the following.

\begin{proposition} \label{pro.PuschForwardVectorbundle}
With notation as above, let $\sN$ be an invertible sheaf on $X$ and suppose that $\sF = \sL_0 \oplus \cdots \oplus \sL_d$. Then, there is an isomorphism
\[
\Lambda \: \bigoplus_{0 \leq i_0, \ldots ,i_d \leq q-1} \varpi^* F^e_* \big(\sL_0^{i_0} \otimes \cdots \otimes \sL_d^{i_d} \otimes \sN \big)  \xrightarrow{\cong} F^e_* \varpi^* \sN. 
\]
In particular, if $\sF=\sL$ is an invertible sheaf then $\bigoplus_{i=0}^{q-1} \varpi^* F^e_* \sL^{i} \xrightarrow{\cong} F^e_* \sO_{\bV(\sL)}$.

\end{proposition}
\begin{proof}
We explain first what the $\sO_{\bV}$-linear map $\Lambda$ is. Since $\varpi$ is affine, we may equivalently specify what the $\sS$-linear map $\varpi_* \Lambda$ is, which corresponds (by the projection formula and the functoriality of Frobenius) to the description of a $\sS$-linear map
\[
\bigoplus_{0 \leq i_0,\ldots ,i_d \leq q-1}  \sS \otimes F^e_* (\sL_0^{i_0} \otimes \cdots \otimes \sL_d^{i_d} \otimes \sN)  \to  F^e_* (\sS \otimes \sN).
\]
Such description is done by taking the direct sum $\bigoplus_{0 \leq i_0,\ldots ,i_d \leq q-1} \bigoplus_j \bigoplus_{j_0 + \cdots + j_d=j}$ over the structural maps
\begin{align*}
\sL_0^{j_0} \otimes  \cdots \otimes \sL_d^{j_d} \otimes F^e_* ( \sL_0^{i_0} \otimes \cdots \otimes \sL_d^{i_d} \otimes \sN)  \xrightarrow{\cong} &F^e_* (\sL_0^{j_0q+i_0} \otimes  \cdots \otimes \sL_d^{j_dq+i_d} \otimes \sN)\\
\to& F^e_* (\sS \otimes \sN).
\end{align*}
where the first arrow is simply the projection formula isomorphism associated to $F^e$. 

Finally, by the euclidean algorithm, we further see why $\varpi_* \Lambda$ and so $\Lambda$ are isomorphisms.
\end{proof}

Recall that $\bZ \to \Pic \bP(\sF)$; $1 \mapsto \sO_{\bP}(1)$, and $\pi^* \: \Pic X \to \Pic \bP(\sF)$ induce an isomorphism $\bZ \oplus \Pic X  \xrightarrow{\cong} \Pic \bP(\sF)$. We then have the following.

\begin{proposition} \label{pro.PushforwardProjectiveBundles}
With notation as above, let $\sN$ be an invertible sheaf on $X$ and $n \in \bZ$. Write $n=kq+m$ with $0 \leq m \leq q-1$. Suppose that $\sF = \sL_0 \oplus \cdots \oplus \sL_d$. Then, there is an isomorphism
\[
\Pi \: \bigoplus_{i=0}^d{\sO_{\bP}(k-i)} \otimes \bigoplus_{\substack{0\leq i_0,\ldots,i_d \leq q-1 \\  i_0+\cdots +i_d = m+iq}} \pi^* F^e_* (\sL_0^{i_0} \otimes \cdots \otimes \sL_d^{i_d} \otimes \sN) \xrightarrow{\cong} F^e_* (\sO_{\bP}(n) \otimes \pi^*\sN) 
\]
In particular, if $\sL_i = \sO_X$ for all $i$ then:
\[
\Pi \: \pi^* F^e_* \sN \otimes \bigoplus_{i=0}^d{\sO_{\bP}(k-i)^{\oplus a(i,m;d,e)}}  \xrightarrow{\cong} F^e_* (\sO_{\bP}(n) \otimes \pi^*\sN) 
\]
where $a(i,m;d,e)$ is the number of combinations of $m+iq$ with $d+1$ parts in the interval $[0,q-1]$. Concretely,
\[
a(i,m;d,e) \coloneqq \sum_{j+k=i}{(-1)^k\binom{d+1}{k}\binom{jq+m+d}{d}}
\]
Likwewise, if $\sL_0=\sL$ and $\sL_1,\ldots, \sL_d=\sO_X$ then:
\[
\Pi \: \bigoplus_{i=0}^d{\sO_{\bP}(k-i)} \otimes \bigoplus_{j=0}^{q-1} \pi^* F^e_* (\sL^{j} \otimes \sN)^{\oplus a(i,m-j;d-1,e)} \xrightarrow{\cong} F^e_* (\sO_{\bP}(n) \otimes \pi^*\sN) 
\]
\end{proposition}
\begin{proof}
Observe that $F^e_* \sS$ is a $\frac{1}{q}\bZ$-graded $\sS$-module (where $F^e$ is the $e$-th Frobenius homomorphism on $\sS$) by declaring elements in $F^e_*\sS_i \subset F^e_*\sS$ to sit in degree $i/q$. Thus, if $x \in \sS_i$ and $F^e_* y \in F^e_* \sS_j \subset (F^e_* \sS)_{j/q}$ then 
\[
xF^e_* y = F^e_* (x^q y) \in F^e_* \sS_{iq+j} \subset (F^e_*\sS)_{i+j/q}.
\] 
Of course, we meant the above description to be on local sections. The same applies for $\sS \otimes \pi^*\sN $ in place of $\sS$. In general, if $\sM$ is a $\frac{1}{q}\bZ$-graded $\sS$-module, we can write
\[
\sM = \bigoplus_{n=0}^{q-1} \bigoplus_{i\in \bZ} \sM_{i+n/q}
\]
where the $\sM^{(n)} \coloneqq \bigoplus_{i\in \bZ} \sM_{i+n/q}$ are $\bZ$-graded $\sS$-modules. In other words, $\sM$ admits a graded direct sum decomposition
\[
\sM = \bigoplus_{n=0}^{q-1} \sM^{(n)},
\]
where $\sM^{(n)}$ is a $\bZ$-graded direct summand of $\sM$. In the particular case $\sM = F^e_* \sS$, we have
\[
(F^e_*\sS)^{(n)} = \bigoplus_{i\in \bZ} F^e_* \sS_{iq+n}.
\]
Note that 
\begin{align*}
\Gamma_*\big(F^e_*(\sO_{\bP}(n) \otimes \pi^*\sN )\big)& = \bigoplus_{i \in \bZ} \pi_* \big( \sO_{\bP}(i) \otimes F^e_*(\sO_{\bP}(n) \otimes \pi^*\sN ) \big) \\
&= \bigoplus_{i\in \bZ} \pi_* \big( F^e_*( \sO_{\bP}(iq+n) \otimes \pi^*\sN ) \big) \\
&=\bigoplus_{i\in \bZ} F^e_* \big( \pi_*( \sO_{\bP}(iq+n)\otimes \pi^*\sN ) \big)\\
&=\bigoplus_{i\in \bZ} F^e_* \big( \pi_*\sO_{\bP}(iq+n) \otimes \sN \big)\\
&=\bigoplus_{i\in \bZ} F^e_* \big( \sS_{iq+n} \otimes \sN \big)\\
&= \big(F^e_*(\sS \otimes \sN)\big)^{(n)},
\end{align*}
where, by an abuse of notation, we wrote equality instead of isomorphism when we applied the projection formula. In particular, we have a natural isomorphism
\[
\Bigl(\big(F^e_*( \sS \otimes \sN)\big)^{(n)} \Bigr)^{\sim} \xrightarrow{\cong} F^e_*( \sO_{\bP}(n) \otimes \pi^*\sN ).
\]
We compute $\big(F^e_*(\sS \otimes \sN)\big)^{(n)}$ next. By \autoref{pro.PuschForwardVectorbundle} and its proof, 
\begin{align*}
F^e_* (\sS \otimes \sN) &\xleftarrow{\cong} \bigoplus_{0\leq i_0,\ldots,i_d \leq q-1} \sS \otimes F^e_* (\sL_0^{i_0} \otimes \cdots \otimes \sL_d^{i_d} \otimes \sN)\\
&= \bigoplus_{n=0}^{q-1} \bigoplus_{ \substack{0\leq i_0,\ldots,i_d \leq q-1\\i_0 + \cdots + i_d \equiv n \bmod q}} \sS \otimes F^e_* (\sL_0^{i_0} \otimes \cdots \otimes \sL_d^{i_d} \otimes \sN),
\end{align*}
as $\sS$-modules. However, by definition, the $\sS$-linear map
\begin{equation}\label{eqn.Grading}
\sS \otimes F^e_* (\sL_0^{i_0} \otimes \cdots \otimes \sL_d^{i_d} \otimes \sN) \to F^e_*(\sS \otimes \sN) 
\end{equation}
becomes graded if we declare
\[
\big( 
\sS \otimes F^e_* (\sL_0^{i_0} \otimes \cdots \otimes \sL_d^{i_d} \otimes \sN) \big)_i \coloneqq \sS_{i-\lfloor (i_0 + \cdots + i_d)/q \rfloor} \otimes F^e_* (\sL_0^{i_0} \otimes \cdots \otimes \sL_d^{i_d} \otimes \sN).
\]
In other words, the $\sS$-linear map \autoref{eqn.Grading} defines a graded homomorphism of $\sS$-modules:
\[
\sS\Bigl(-\bigl\lfloor (i_0+ \cdots + i_d)/q \bigr\rfloor\Bigr) \otimes F^e_* (\sL_0^{i_0} \otimes \cdots \otimes \sL_d^{i_d} \otimes \sN) \to F^e_*(\sS \otimes \sN) 
\]
In conclusion,
\[
\big(F^e_*(\sS \otimes \sN)\big)^{(n)} \xleftarrow{\cong}\bigoplus_{\substack{0\leq i_0,\ldots,i_d \leq q-1 \\  i_0+\cdots +i_d \equiv n \bmod q}} \sS\left(-\frac{i_0+\cdots+i_d-n}{q}\right)\otimes F^e_* (\sL_0^{i_0} \otimes \cdots \otimes \sL_d^{i_d} \otimes \sN).
\]
as $\bZ$-graded $\sS$-modules. 

Putting everything together:
\[
\Pi \: \bigoplus_{\substack{0\leq i_0,\ldots,i_d \leq q-1 \\  i_0+\cdots +i_d \equiv n \bmod q}} \sO_{\bP}\left(-\frac{i_0+\cdots+i_d-n}{q}\right)\otimes \pi^*F^e_* (\sL_0^{i_0} \otimes \cdots \otimes \sL_d^{i_d} \otimes \sN) \xrightarrow{\cong} F^e_* (\sO_{\bP}(n)\otimes \pi^* \sN)
\]
On the other hand, the values of $i_0 + \cdots +i_d $ congruent to $n$ modulo $q$ subject to $0 \leq i_0,\ldots,i_d \leq q-1$ are $m,m+q,m+2q,\ldots,m+dq$. Therefore,
\[
\Pi \: \bigoplus_{i=0}^d{\sO_{\bP}(k-i)} \otimes \bigoplus_{\substack{0\leq i_0,\ldots,i_d \leq q-1 \\  i_0+\cdots +i_d = m+iq}} \pi^* F^e_* (\sL_0^{i_0} \otimes \cdots \otimes \sL_d^{i_d} \otimes \sN) \xrightarrow{\cong} F^e_* (\sO_{\bP}(n) \otimes \pi^*\sN) 
\]
Of course, if $\sL_0,\ldots,\sL_d = \sO_{X}$, then 
\[
\bigoplus_{\substack{0\leq i_0,\ldots,i_d \leq q-1 \\  i_0+\cdots +i_d = m+iq}} \pi^* F^e_* (\sL_0^{i_0} \otimes \cdots \otimes \sL_d^{i_d} \otimes \sN) = (\pi^*F^e_*\sN)^{\oplus a(i,m;d,e)}
\]
where $a(i,m;d,e)$ is the number of combinations of $m+iq$ with $d+1$ parts in the interval $[0,q-1]$. In other words, $a(i,m;d,e)$ is the coefficient of $t^{m+iq}$ in the following power series
\begin{align*}
\bigl(1+t+t^2+\cdots +t^{q-1}\bigr)^{d+1} = \left(\frac{1-t^q}{1-t}\right)^{d+1}&=(1-t^q)^{d+1}\cdot \sum_{l \geq 0}{\binom{l+d}{d}t^l} \\
&=\sum_{\substack{0\leq k \leq d+1 \\ l \geq 0}}{(-1)^k\binom{d+1}{k}\binom{l+d}{d} t^{kq+l} } \\
& = \sum_{\substack{0\leq k \leq d+1 \\ 0 \leq m \leq q-1 \\ j \geq 0}}{(-1)^k\binom{d+1}{k}\binom{jq+m+d}{d} t^{m+(k+j)q} }
\end{align*}
Therefore,
\[
a(i,m;d,e) = \sum_{j+k=i}{(-1)^k\binom{d+1}{k}\binom{jq+m+d}{d}} = \sum_{j=0}^i{(-1)^{i-j}\binom{d+1}{i-j}\binom{jq+m+d}{d}}.
\]
The proposition then follows.
\end{proof}

\begin{remark} Given $d,e \in \bN \smallsetminus \{0\}$ and $i,m \in \bZ$, the integer $a(i,m;d,e) \neq 0$ if and only if $0 \leq m+iq \leq (d+1)(q-1)$. For instance, $a(d,q-1;d,e) = 0$ and
\[
a(i,0;d,e) = \sum_{j=0}^i{(-1)^{i-j}\binom{d+1}{i-j} \binom{jq+d}{d}} \neq 0 \text{ if and only if } i=0,\ldots,d.
\]
An easy case worth keeping in mind is $a(0,m;d,e)= \binom{m+d}{d}$. In general, observe that $a(i,m;d,e)$ is a polynomial in $q$ of degree $d$. Its leading coefficient can be computed as follows. First, we note that
\[
\lim_{e \rightarrow \infty}\frac{a(i,0;d,e)}{q^d/d!} = \sum_{j=0}^i{(-1)^{i-j}\binom{d+1}{i-j} j^d } = \sum_{j=0}^i{(-1)^{j}\binom{d+1}{j} (i-j)^d } = A(d,i)
\]
where the $A(d,i)$ are the so-called \emph{Eulerian numbers}; see \cite[\S 3]{SinghFSignatureOfAffineSemigroup} and the references therein. Thus, the leading coefficient of $a(i,0;d,e)$ is $A(d,i)/d!$, which turns out to be the $F$-signature of the cone singularity defined by the Segre embedding of $\bP^{i-1} \times \mathbb{P}^{(d-1)-(i-1)}$. More generally,
\[
\lim_{e \rightarrow \infty}\frac{a(i,m;d,e)}{q^d/d!} = \sum_{j=0}^i{(-1)^{i-j}\binom{d+1}{i-j} (j+m/q)^d }.
\]
Of course, one verifies that $\sum_{i=0}^d{a(i,m;d,e)} = q^d$, $\sum_{i=0}^d {A(d,i)} = d!$. 
\end{remark}

\begin{question}
Can we generalize \autoref{pro.PushforwardProjectiveBundles} to the case where $\sF$ is not fully decomposable? For instance, can we describe the case of elliptic ruled surfaces, i.e., the case where $\sF$ is a indecomposable rank-$2$ locally free sheaf on an elliptic $C$? In such a case, $\sF$ is necessarily an extension of invertible sheaves; see \cite[V, Corollary 2.7, Exercise 3.3]{Hartshorne}. The main difficulty seems to be finding a nice description of $\sS_i=\Sym^i \sF$. It is unclear, at least for the authors, whether \cite[II, Exercise 5.16.(c)]{Hartshorne} is good enough for such a purpose. It is worth noting that we work out very explicitly the case of Hirzebruch surfaces in \autoref{ex.HirzebruchSurfaces} below. Already in this much simpler case, we can note some complexity emerging.
\end{question}

\begin{corollary}
On $X=\bP^d$, we then see that
\[
\sE_e = \sB_e^{1,\vee} \cong \bigoplus_{i=1}^d \sO(i)^{\oplus a(i,0;d,e)}
\quad \text{and} \quad
\sE_e^{\vee} \otimes \omega^{-1} = \sB_e^{d,\vee} \cong  \bigoplus_{i=1}^d \sO(d+1-i)^{\oplus a(i,0;d,e)}
\]
are both ample.
\end{corollary}

\subsection{Graded Kunz's Theorem} \label{sec.GradedKUNZ}

Let $S = \bigoplus_{i \geq 0} S_i$ be an $\bN$-graded ring that is finitely generated by $S_1$ over $S_0=\kay$. Let us set $\fram =S_{+} = \bigoplus_{i \geq 1} S_i$. The following basic observations were made in the proof \autoref{pro.PushforwardProjectiveBundles}. $F^e_* S = \bigoplus_{i \in \bN} F^e_* S_i$ is a $\frac{1}{q}\bZ$-graded $S$-module by declaring that the summand $F^e_*S_i$ sits in degree $i/q$. Thus, if $x \in S_i$ and $F^e_* y \in F^e_*S_j \subset (F^e_* S)_{j/q} $ then $xF^e_* y = F^e_* (x^q y) \in F^e_* S_{iq+j} \subset (F^e_*S)_{i+j/q}$. In general, if $M$ is $\frac{1}{q}\bZ$-graded $S$-module, we may write $M = \bigoplus_{n=0}^{q-1} \bigoplus_{i \in \bN} M_{i+n/q}$ where the $M^{(n)} \coloneqq \bigoplus_{i \in \bN} M_{i+n/q}$ are $\bN$-graded $S$-modules. In other words, $M$ admits a graded direct sum decomposition $M = \bigoplus_{n=0}^{q-1} M^{(n)}$, and so $M^{(n)}$ is a graded direct summand of $M$. In the case $M = F^e_* S$, we have $
(F^e_*S)^{(n)} = \bigoplus_{i \in \bN} F^e_* S_{iq+n}$. Let $X = \Proj S$. We readily see that $
\Gamma_*\big(F^e_* \sO_X(n)\big) = (F^e_*S)^{(n)}$. In particular, $F^e_*\sO_X(n) \cong \big((F^e_*S)^{(n)}\big)^{\sim}$. Hence, the following graded version of Kunz's theorem holds.

\begin{proposition}
With notation as above, $S$ is regular if and only if $(F^e_*S)^{(n)}$ is free as a graded $S$-module for all $n=0, \ldots, q-1$.
\end{proposition}
\begin{proof}
If $(F^e_*S)^{(n)}$ is free as a graded $S$-module for all $n=0, \ldots, q-1$, then it is free as an ordinary $S$-module and so is $F^e_*S$. Hence, $S$ is regular by Kunz's theorem. 

Conversely, if $S$ is regular, then $F^e_*S$ is a projective $S$-module according to Kunz's Theorem. Therefore, $(F^e_*S)^{(n)}$ is a direct summand of a projective $S$-module. Hence, $(F^e_*S)^{(n)}$ is a projective $\bZ$-graded $S$-module. Nonetheless, projective graded $S$-modules are free graded $S$-modules.
\end{proof}

We recall the following well-known statements and prove them for the sake of completeness.

\begin{lemma}
With notation as above, $S$ is regular if and only if $S$ is isomorphic to the standard graded polynomial ring over $\kay$. Indeed, if $S_{\fram}$ is a regular local ring, then $S \cong  \kay[x_1, \ldots , x_d]$ as graded rings, where $d = \dim_{\kay} S_{\fram} / \fram S_{\fram} = \dim S_{\fram} $.
\end{lemma}
\begin{proof}
By \cite[Theorem 13.8 (iii)]{MatsumuraCommutativeRingTheory}, even without assuming that $S_{\fram}$ is a regular local ring, $S$ is isomorphic (as graded rings) to the graded associated ring of the local ring $S_{\fram}$ (with respect to its maximal ideal). On the other hand, if $S_{\fram}$ is such that $\fram S_{\fram}$ is generated by a regular sequence, then its associated graded ring is a standard graded polynomial ring; see \cite[Theorem 1.1.8]{BrunsHerzog}.  
\end{proof}
The following well-known characterization of projective spaces is then obtained.
\begin{corollary}
The projective spaces are the only ones that admit a regular ring of sections. More precisely, if $X$ is a variety that admits an ample invertible sheaf $\sA$ so that the corresponding ring of sections $R(X,\sA)$ is regular, then $X$ is isomorphic to $\bP^{\dim X}$.
\end{corollary}

We say that a locally free sheaf on a scheme $X$ is \emph{$\sL$-split}, for a given invertible sheaf $\sL$, if it is isomorphic to a direct sum of invertible sheaves whose class in $\Pic X$ belongs to $\langle \sL \rangle_{\bZ} \subset \Pic X$. Applying the graded Kunz theorem, we obtain the following.

\begin{corollary}\label{cor.GradedKunzTheorem}
Let $X$ be a variety admitting an ample invertible sheaf $\sA$ such that $F^e_* \sA^{n}$ is \emph{$\sA$-split} for all $n=0,\ldots, q-1$ (for some $0\neq e \in \bN$). Then, $X \cong \bP^{\dim X}$.
\end{corollary}
\begin{proof}
Set $S = \Gamma_* \sA = R(X,\sA)$. Then $F^e_*S = \bigoplus_{n=0}^{q-1} \Gamma_*(F^e_* \sA^n)$. Our hypothesis then says that $\Gamma_*(F^e_* \sA^n)=(F_*^e S)^{(n)}$ is a free graded $S$-module as $\Gamma_*(\sA^{a}) = S(a)$.
\end{proof}

\section{Examples} \label{sec.Examples}

This section aims to describe some further examples of Frobenius pushforwards of invertible sheaves (and so of $\sE_e$) on some basic varieties. Our main motivation is to use them in the proofs of our main theorems in \autoref{sec.MainSection}. For example, \autoref{ex.HirzebruchSurfaces} gives an alternative proof of \autoref{cor.MainCorDim2}, and the computations of both \autoref{sec.BlowupProjSpace} and \autoref{sec.CONESEXAMPLES} are essentially used in the proof of \autoref{prop.RulingOutCases}.

We commence with those examples that can be easily obtained from the formulas in \autoref{pro.PushforwardProjectiveBundles} and \autoref{pro.PuschForwardVectorbundle} (including singular ones such as those in \autoref{sec.CONESEXAMPLES}). The authors are aware that some of our examples are toric varieties, whose Frobenius pushforwards have been greatly described in \cite{AchingerCharacterizationOfToricVarieties,ThomsenFrobeniusDirectImagesOfLineBundlesToricVarieties}. For instance, toric varieties are characterized as those varieties whose Frobenius pushforwards of invertible sheaves split as a direct sum of invertible sheaves.

More generally, there is a hearty body of works describing the Frobenius pushforward of the structure sheaf of certain homogeneous spaces in the context of $\sD$-affinity and representation theory. See, for instance, \cite{KanedaFrobeniusPushforwardsStructureSheafHomogeneousSpaces,RaedscheldersSpenkoVandenBerghFrobeniusMorphismInInvariantTheory,SamokhinFrobeniusMorohismFlagVarietiesI,SamokhinFrobeniusMorohismFlagVarietiesII} and the references therein. It would be very interesting to use these computations to analyze the positivity of $\sE_X$; this will be pursued elsewhere.

For another set of interesting examples, we recommend \cite{SannaiTanakaOrdinaryAbelianFrobenius, EjiriSannaACharacterizationOrdnaryAblianVarieties, HaraFrobeniusSummandsBlownUpSurfaceofP2}. The former two papers are concerned with (ordinary) abelian varieties (which shall not concern us since these are not Fano) whereas the latter is concerned with the degree-$5$ del Pezzo surface. It is worth noting that in those works the emphasis has been on the (in)decomposability of Frobenius pushforwards, whereas our focus is on positivity. In this section, we fix $0 \neq e \in \bN$.

\subsection{Products of projective spaces} \label{ex.ProducProjectiveSpaces}
 By direct application of \autoref{pro.PushforwardProjectiveBundles}: 
\[
\bigoplus_{i=0}^d \sO_{\bP^d}(k-i)^{\oplus a(i,m;d,e)} \to F^e_* \sO_{\bP^d} (n) 
\]
where $n=kq+m$, $0 \leq m \leq q-1$. In particular, if $d=1$:
\[
\sO_{\bP^1}(k)^{\oplus (m+1)} \oplus \sO_{\bP^1}(k-1)^{\oplus(q-1-m)} \xrightarrow{\cong}  F^e_* \sO_{\bP^1} (n). 
\]
For $d=2$, we have that $F^e_* \sO_{\bP^2} (n)$ is isomorphic to:
\[
\sO_{\bP^2}(k)^{\oplus \frac{(m+1)(m+2)}{2}} \oplus \sO_{\bP^2}(k-1)^{\oplus \frac{q^2+(2m+3)q-2(m+1)(m+2)}{2}} \oplus \sO_{\bP^2}(k-2)^{\oplus \frac{(q-(m+1))(q-(m+2))}{2} }
\]

Let $X = \bP^r \times \bP^s$, consider its canonical projections $\pi_r \: X \to \bP^r$ and $\pi_s \: X \to \bP^s$, and set $\sO(u,v) \coloneqq \pi_r^* \sO_{\bP^r}(u) \otimes \pi_s ^* \sO_{\bP^s}(v)$. Then, writing $u = kq +m$ and $v=lq+n$ with $0\leq m,n \leq q-1$, we have:
\begin{align*}
F^e_* \sO(u,v) &\cong \Biggl( \bigoplus_{i=0}^r{\pi_r^*\sO_{\bP^d}(k-i)^{\oplus a(i,m;r,e)}} \Biggl) \otimes \Biggl( \bigoplus_{j=0}^s{\pi_s^*\sO_{\bP^d}(l-j)^{\oplus a(i,n;s,e)}} \Biggl) \\
& = \bigoplus_{\substack{0\leq i \leq r \\ 0 \leq j \leq s }}{\sO(k-i,l-j)^{\oplus a(i,m;r,e)a(j,n;s,e)}}. 
\end{align*}
In particular, $\sum_{\substack{0\leq i \leq r \\ 0 \leq j \leq s }}{a(i,m;r,e)a(j,n;s,e)} = q^{r+s}$. Moreover,
\[
F^e_* \sO \cong \bigoplus_{\substack{0\leq i \leq r \\ 0 \leq j \leq s }}{\sO(-i,-j)^{\oplus a(i,0;r,e)a(j,0;s,e)}} \quad \text{and} \quad \sE_{e} \cong \bigoplus_{\substack{0\leq i \leq r \\ 0 \leq j \leq s \\0<i+j}}{\sO(i,j)^{\oplus a(i,0;r,e)a(j,0;s,e)}}
\]
Hence, $\sE_{e,X}$ is nef yet not ample as it contains $\sO(0,1)$ and $\sO(1,0)$ as direct summands.
\subsection{Hirzebruch surfaces} \label{ex.HirzebruchSurfaces}
Let $C \coloneqq \bP^1$, $\sF_{\varepsilon} \coloneqq \sO_C(-\varepsilon) \oplus \sO_C $ with $\varepsilon \in \bN$,\footnote{We use $\varepsilon$ instead of $e$ to avoid any confusion with the exponent of Frobenius.} and $\pi \: X_{\varepsilon} \to C$ be the corresponding projective bundle. In what follows, we use \cite[V, Notation 2.8.1]{Hartshorne}. That is, $C_0 $ will denote the section of $\pi$ given by $\sF_{\varepsilon} \to \sO_C(-\varepsilon) \to 0$ (thus $\sO_{X_{\varepsilon}}(1)\cong \sO_{X_{\varepsilon}}(C_0)$), and $f$ is the fiber of $\pi_{\varepsilon}$ along a chosen point representing the divisor class of $\sO_C(1)$. Recall that $X_{\varepsilon}$ can be thought of as the blowup at the vertex singularity of the projective cone defined by the Veronese embedding of $\bP^1$. Indeed, the complete linear system $|C_0 + \varepsilon f|$ defines the blowup morphism. Under such description, $C_0$ is none other than the exceptional divisor and, letting $C_1$ denote the section of $\pi$ corresponding to the quotient $\sF_{\varepsilon} \to \sO_C \to 0$, we have the linear equivalence $C_1 \sim C_0 + \varepsilon f$ (which is the pullback of a ruling of the projective cone). See \cite[V, Theorem 2.17]{Hartshorne}. Let $u,v \in \bZ$, and write $u = kq +m$ with $0\leq m \leq q-1$. Applying \autoref{pro.PushforwardProjectiveBundles} yields:
\begin{align*}
&F^e_* \sO_{X_{\varepsilon}}(uC_0+vf)\\
\cong{} &\Biggl(\sO_{X_{\varepsilon}}(kC_0) \otimes \bigoplus_{j=0}^m \pi^*F^e_* \sO_C(v-j\varepsilon )\Biggr)\oplus \Biggl(\sO_{X_{\varepsilon}}((k-1)C_0) \otimes \bigoplus_{j={m+1}}^{q-1} \pi^*F^e_* \sO_C(v-j\varepsilon)\Biggr)\\
\cong{} &\sF_1 \oplus \sF_2 \oplus \sF_3 \oplus \sF_4,
\end{align*}
where:
\begin{align*}
\sF_1 &= \bigoplus_{j=0}^{m} \sO_{X_{\varepsilon}}\big(kC_0+\lfloor (v-j\varepsilon)/q \rfloor f \big)^{\oplus \big([v-j\varepsilon]_q+1\big) }\\
\sF_2 &= \bigoplus_{j=0}^{m} \sO_{X_{\varepsilon}}\big(kC_0+(\lfloor (v-j\varepsilon)/q \rfloor - 1) f \big)^{\oplus\big(q-1-[v-j\varepsilon]_q\big) }\\
\sF_3 &= \bigoplus_{j=m+1}^{q-1} \sO_{X_{\varepsilon}}\big((k-1)C_0+\lfloor (v-j\varepsilon)/q \rfloor f \big)^{\oplus\big([v-j\varepsilon]_q+1\big) }\\
\sF_4 &= \bigoplus_{j=m+1}^{q-1} \sO_{X_{\varepsilon}}\big((k-1)C_0+(\lfloor (v-j\varepsilon)/q \rfloor - 1) f \big)^{\oplus\big(q-1-[v-j\varepsilon]_q\big) }
\end{align*}
In particular, setting $u,v = 0$:
\begin{align*}
F^e_* \sO_{X_{\varepsilon}} \cong{} &\sO_{X_{\varepsilon}} \oplus \sO_{X_{\varepsilon}}(-f)^{\oplus (q-1)} \oplus \bigoplus_{j=1}^{q-1} \sO_{X_{\varepsilon}}\big(-C_0+\lfloor -j\varepsilon/q \rfloor f\big)^{\oplus\big([-j\varepsilon]_q+1\big)}\\
&\oplus \bigoplus_{j=1}^{q-1} \sO_{X_{\varepsilon}}\big(-C_0+(\lfloor -j\varepsilon/q \rfloor - 1 ) f\big)^{\oplus\big(q-1-[-j\varepsilon]_q\big)}
\end{align*}

\subsubsection{Case $\varepsilon =1$} Specializing to $\varepsilon = 1$ gives the blowup of $\bP^2$ at a point. In this case,
\begin{align*}
&F^e_* \sO_{X_{1}}\\ \cong {} &\sO_{X_{1}} \oplus \sO_{X_{1}}(-f)^{\oplus (q-1)} \oplus \bigoplus_{j=1}^{q-1} \sO_{X_{1}}(-C_0- f)^{\oplus(q-j+1)}\oplus \bigoplus_{j=1}^{q-1} \sO_{X_{1}}(-C_0- 2f)^{\oplus (j-1)}\\
\cong{} &\sO_{X_{1}} \oplus \sO_{X_{1}}(-f)^{\oplus (q-1)} \oplus \sO_{X_{1}}(-C_0- f)^{\oplus \frac{(q+2)(q-1)}{2}}\oplus \sO_{X_{1}}(-C_0- 2f)^{\oplus \frac{(q-2)(q-1)}{2}}
\end{align*}
Equivalently, using $C_1 \sim C_0 + f$, we may write:
\[
F^e_* \sO_{X_{1}}\cong \sO_{X_{1}} \oplus \sO_{X_{1}}(C_0-C_1)^{\oplus (q-1)} \oplus \sO_{X_{1}}( -C_1)^{\oplus \frac{(q+2)(q-1)}{2}}\oplus \sO_{X_{1}}(C_0- 2C_1)^{\oplus \frac{(q-2)(q-1)}{2}}
\]
Pulling this back to $X_1 \smallsetminus C_0$ recovers $F^e_* \sO_{\bP^2}$ and pulling it back along $\bP^1 \cong C_0 \to X_1$ yields:
\[
\sO_{\bP^1}^{\oplus \frac{q(q+1)}{2}} \oplus \sO_{\bP^1}(-1)^{\oplus \frac{q(q-1)}{2}},
\]
which implies that $\sE_{e,X_1}$ is not ample.

Let $v=lq+n$ with $0\leq n \leq q-1$. We may easily compute $F^e_* \sO_{X_1}(uC_0 + vf)$. However, this will depend on whether $m \leq n$ or $m>n$. If $m\leq n$, $F^e_* \sO_{X_1}(uC_0 + vf)$ is isomorphic to
\begin{align*}
&\sO_{X_1}(kC_0 + lf)^{\oplus \frac{(m+1)(m+2+2(n-m))}{2} } \oplus  \sO_{X_1}(kC_0 + (l-1)f)^{\oplus \frac{(m+1)(2q-(m+2)-2(n-m))}{2}}  \\
&\oplus \sO_{X_1}((k-1)C_0+lf)^{\oplus \frac{(n-m)(n-m+1)}{2}} \\
&\oplus  \sO_{X_1}((k-1)C_0 + (l-1)f)^{\oplus \frac{(q-n-1)(q+n+2)-(n-m)(2q-n+m-1)}{2} } \\
&\oplus  \sO_{X_1}((k-1)C_0 + (l-2)f)^{\oplus \frac{(q-n-1)(q-n-2)}{2}}.
\end{align*}
If $n>m$, one has a similar description, but this time the invertible sheaves showing up as direct summands are $\sO_{X_1}(kC_0+lf)$, $\sO_{X_1}(kC_0+(l-1)f)$, $\sO_{X_1}(kC_0+(l-2)f)$, $\sO_{X_1}((k-1)C_0+(l-1)f)$, and $\sO_{X_1}((k-1)C_0+(l-2)f)$.

\subsubsection{Case $\varepsilon = 2$} This corresponds to the blowup of the singular quadric cone at its vertex. \emph{Let us assume first $p \neq 2$.} If $1 \leq j \leq (q-1)/2$ then $q-1 \geq q-2j \geq 1$ and so $\lfloor -2j/q \rfloor = -1$, $[-2j]_q=q-2j$. On the other hand, if $(q+1)/2 \leq j \leq q-1$ then $q-1 \geq 2q-2j \geq 2$ and so $\lfloor -2j/q \rfloor = -2$, $[-2j]_q=2q-2j$. Therefore, 
\begin{align*}
&\bigoplus_{j=1}^{q-1} \sO_{X_{2}}\big(-C_0+\lfloor -2j/q \rfloor f\big)^{\oplus\big([-2j]_q+1\big)}\\
={} &\bigoplus_{j=1}^{(q-1)/2} \sO_{X_2}(-C_0-f)^{\oplus(q-2j+1)} \oplus \bigoplus_{j=(q+1)/2}^{q-1} \sO_{X_2}(-C_0-2f)^{\oplus(2q-2j+1)}\\
={} & \sO_{X_2}(-C_0-f)^{\oplus\left(\frac{q-1}{2}\right)\left( \frac{q+1}{2} \right)} \oplus \sO_{X_2}(-C_0-2f)^{{\oplus}\left(\frac{q-1}{2}\right)\left( \frac{q+3}{2} \right)}.
\end{align*}
Likewise,
\begin{align*}
 &\bigoplus_{j=1}^{q-1} \sO_{X_{\varepsilon}}\big(-C_0+(\lfloor -j\varepsilon/q \rfloor - 1 ) f\big)^{\oplus\big(q-1-[-j\varepsilon]_q\big)} \\
 ={} &\sO_{X_2}(-C_0-2f)^{\oplus \left( \frac{q-1}{2} \right)^2} \oplus \sO_{X_2}(-C_0-3f)^{ \oplus \left( \frac{q-1}{2} \right) \left(\frac{q-3}{2} \right)} .
\end{align*}
Hence,
\begin{align*}
F^e_* \sO_{X_2} \cong{} &\sO_{X_2} \oplus \sO_{X_2}(-f)^{\oplus (q-1)}  \oplus \sO_{X_2}(C_0-f)^{\oplus \left(\frac{q-1}{2}\right)\left( \frac{q+1}{2} \right)} \\
&\oplus \sO_{X_2}(C_0-2f)^{\oplus \left(\frac{q-1}{2}\right)\left( \frac{2q+2}{2} \right)} \oplus \sO_{X_2}(C_0-3f)^{\oplus \left( \frac{q-1}{2} \right) \left(\frac{q-3}{2} \right)}  . 
\end{align*}

\emph{Let us assume now $p=2$}. There are two cases depending on whether $j\in \{1,\ldots,2^{e-1}\}$ or $j\in \{2^{e-1}+1,\ldots, 2^e-1\}$. In the former case $\lfloor -2j/2^e\rfloor = -1$ and $[-2j]_{2^e}=2^e-2j$, while in the latter case $\lfloor -2j/2^e\rfloor = -2$ and $[-2j]_{2^e}=2^{e+1}-2j$. Thus:

\begin{align*}
&\bigoplus_{j=1}^{q-1} \sO_{X_{2}}\big(-C_0+\lfloor -2j/q \rfloor f\big)^{\oplus\big([-2j]_q+1\big)}\\
={} & \sO_{X_2}(-C_0-f)^{\oplus 2^{2(e-1)}} \oplus \sO_{X_2}(-C_0-2f)^{\oplus(2^{2(e-1)}-1)},
\end{align*}
and 
\begin{align*}
 &\bigoplus_{j=1}^{q-1} \sO_{X_{\varepsilon}}\big(-C_0+(\lfloor -j\varepsilon/q \rfloor - 1 ) f\big)^{\oplus\big(q-1-[-j\varepsilon]_q\big)} \\
 ={} &\sO_{X_2}(-C_0-2f)^{\oplus 2^{2(e-1)}} \oplus \sO_{X_2}(-C_0-3f)^{ \oplus (2^{e-1}-1)^2 } .
\end{align*}
Therefore,
\begin{align*}
F^e_* \sO_{X_2} \cong{} &\sO_{X_2} \oplus \sO_{X_2}(-f)^{\oplus (q-1)}  \oplus \sO_{X_2}(C_0-f)^{\oplus(q/2)^2} \\
&\oplus \sO_{X_2}(C_0-2f)^{\oplus (q^2-2)/2} \oplus \sO_{X_2}(C_0-3f)^{\oplus \left(\frac{q-2}{2} \right)^2}  
\end{align*}
where $q=2^e$.

\subsubsection{Case $\varepsilon =3$} There are three cases depending on whether $q \equiv 0,1,2 \mod 3$. \emph{Suppose first $q \equiv 1 \bmod 3$.} Then, we may write a partition
\[
\{1,\ldots, q-1\} = \{1,\ldots, (q-1)/3\} \cup \{(q+2)/3,\ldots, 2(q-1)/3\} \cup \{(2q+1)/3, \ldots, q-1\},
\]
and we denote these subsets by $J_1$, $J_2$, and $J_3$; respectively. Thus, if $j \in J_i$ then $\lfloor - 3j/q \rfloor = -i$ and $[3j]_{q}=iq-3j$. In particular, just as before, we get:
\begin{align*}
F^e_* \sO_{X_3} \cong{} &\sO_{X_3} \oplus \sO_{X_3}(-f)^{\oplus(q-1)} \oplus \\
& \sO_{X_3}(-C_0-f)^{\oplus \sigma_1} \oplus \sO_{X_3}(-C_0-2f)^{\oplus \sigma_2} \oplus \sO_{X_3}(-C_0-3f)^{\oplus \sigma_3} \oplus \sO_{X_3}(-C_0-4f)^{\oplus \sigma_4},   
\end{align*}
where the exponents $\sigma_i$ are obtained as follows:
\[
    \sigma_1 = \sigma_1',\quad
    \sigma_2 = \sigma_2'+\sigma_1'',\quad
    \sigma_3 = \sigma_3' + \sigma_2'',\quad
    \sigma_4 = \sigma_3'',
\]
where
\begin{align*}
    \sigma_i' &\coloneqq \sum_{j \in J_i} (iq-3j+1) =iq|J_i| - \sum_{j\in J_i}(3j-1) ,\\
    \sigma_i'' &\coloneqq \sum_{j \in J_i} \big((1-i)q + 3j-1\big) = -(i-1) q  |J_i| + \sum_{j\in J_i}(3j-1),
\end{align*}
where  $|J_i|=(q-1)/3$. Then a direct computation shows:
\[
 \sigma_1 = \frac{q(q-1)}{6}, \quad \sigma_2 = \frac{(q+1)(q-1)}{3},\quad \sigma_3 = \frac{(q+1)(q-1)}{3},\quad
\sigma_4 = \frac{(q-4)(q-1)}{6}.
\]

\emph{Let us suppose now $q \equiv 2 \bmod 3$ but $q \neq 2$.}\footnote{The case $q=2$ is trivial for all $\varepsilon$. In the case $\varepsilon = 3$, we have $\sigma_1, \sigma_2 = 1$ and $\sigma_3,\sigma_4= 0$.} Then the same description as above holds but this time using the (asymmetric) partition:
\[
\{1,\ldots, q-1\} = \left\{1,\ldots, \frac{q-2}{3}\right\} \cup \left\{\frac{q+1}{3},\ldots, \frac{2(q-2)}{3}, \frac{2q-1}{3}\right\} \cup \left\{\frac{2q+2}{3},  \ldots, q-2, q-1\right\}
\]
where $|J_1| = (q-2)/3 = |J_3|$ whereas $|J_2|=(q+1)/3$.
In that case, we get:
\[
\sigma_1 = \frac{(q+1)(q-2)}{6},\quad
\sigma_2 = \frac{q^2+2}{3},\quad
\sigma_3 = \frac{(q+2)(q-2)}{3},\quad
\sigma_4 = \frac{(q-3)(q-2)}{6}.
\]
\emph{The final case is $q \equiv 0 \bmod 3$.} If $q=3$, one readily verifies $\sigma_1=1$, $\sigma_2=3$, $\sigma_3=2$, and $\sigma_4=0$. If $q \geq 9$, then one uses the partition
\[
\{1,\ldots, q-1\} = \{1, \ldots, q/3\} \cup \{q/3+1,  \ldots , 2q/3\} \cup \{2q/3+1, \ldots, q-1\}
\]
to obtain, via similar computations, the following exponents:
\[
\sigma_1 = \frac{q(q-1)}{6},\quad
\sigma_2 = \frac{q^2}{3},\quad
\sigma_3 = \frac{q^2-3}{3},\quad
\sigma_4 = \frac{(q-3)(q-2)}{6}.
\]

\subsubsection{General case} For general $\varepsilon$, \emph{let us suppose $q \geq \varepsilon$}. Let us write the following partition
\[
\{1,\ldots,q-1\} = J_1 \cup \cdots \cup J_{\varepsilon-1} \cup J_{\varepsilon},
\]
where
\[
J_i \coloneqq \left\{\left\lfloor (i-1)q/\varepsilon \right\rfloor + 1, \ldots , \left\lfloor iq/\varepsilon \right\rfloor\right\}
\]
if $i=1,\ldots, \varepsilon - 1$, and 
\[
J_{\varepsilon} \coloneqq \left\{\left\lfloor (\varepsilon-1)q/\varepsilon \right\rfloor + 1, \ldots , q-1\right\}.
\]
Thus, if $j \in J_i$ then $\lfloor -\varepsilon j/q\rfloor = -i$ and $[-\varepsilon j]_q=iq-\varepsilon j$. Hence, we may define
\[
    \sigma_i' \coloneqq iq |J_i| -\sum_{j \in J_i}(\varepsilon j-1),\qquad
    \sigma_i''\coloneqq -(i-1)q|J_i| + \sum_{j \in J_i}(\varepsilon j-1),
\]
for $i=1,\ldots, \varepsilon$, and further:
\[
\sigma_1 \coloneqq \sigma_1', \quad \sigma_i \coloneqq \sigma_i' + \sigma_{i-1}'',\quad \sigma_{\varepsilon+1} \coloneqq \sigma_{\varepsilon}'',
\]
for $i=2,\ldots, \varepsilon$. Then,
\[
F^e_* \sO_{X_{\varepsilon}} \cong \sO_{X_{\varepsilon}} \oplus \sO_{X_{\varepsilon}}(-f)^{\oplus(q-1)} \oplus \bigoplus_{i=1}^{\varepsilon +1} \sO_{X_{\varepsilon}}(-C_0-if)^{\oplus \sigma_i}. 
\]
Computing $\sigma_i$ is rather subtle. To do so, we need to consider the arithmetic modulo $\varepsilon$. Precisely, for $k, l \in \{0, \ldots, \varepsilon -1\}$, let $\rho_{k,l} \in \{0, \ldots , \varepsilon - 1\}$ be the residue of $kl$ modulo $\varepsilon$. For convenience, we also define $\rho_{k,\varepsilon} \coloneqq \varepsilon$. The point is that, for $i=0,\ldots, \varepsilon -1$, we have $\lfloor iq/\varepsilon \rfloor = (iq-\rho_{k,i})/\varepsilon$ if $k$ is the residue of $q$ modulo $\varepsilon$. Further, one has
\[
|J_i| = \frac{q-\rho_{k,i}+\rho_{k,i-1}}{\varepsilon}, \quad \sum_{j \in J_1 \cup \cdots \cup J_i}{(\varepsilon j - 1)} = \frac{(iq-\rho_{k,i})(iq-\rho_{k,i}+\varepsilon-2)}{2 \varepsilon} 
\]
for $i=1,\ldots, \varepsilon$. Thus, a lengthy, direct computation shows:
\begin{align*}
    \sigma_1 &= \frac{(q-\rho_{k,1})(q+\rho_{k,1}-\varepsilon+2)}{2\varepsilon},\\
    \sigma_i &= \frac{q^2-\big(\rho_{k,i}^2-2\rho_{k,i-1}^2+\rho_{k,i-2}^2-(\varepsilon-2)(\rho_{k,i}-2\rho_{k,i-1}+\rho_{k,i-2})\big)/2}{\varepsilon}, \quad i=2,\ldots,\varepsilon,\\
    \sigma_{\varepsilon+1} & = \frac{(q-\varepsilon+\rho_{k, \varepsilon-1})(q-\rho_{k,\varepsilon-1}-2)}{2 \varepsilon}
\end{align*}
where $k$ is the residue of $q$ modulo $\varepsilon$. Of course, there are two very simple cases. Namely, $k=0$ (i.e. $\varepsilon=p$) and $k=1$ (e.g. $p \equiv 1 \bmod \varepsilon$) for $\rho_{0,i}=0$ and $\rho_{1,i} = i$ for $i=0,\ldots , \varepsilon -1$.

Note that the pullback of $\sO_{X_{\varepsilon}}(f)$ and $\sO_{X_{\varepsilon}}(C_0)$ to $C_0$  correspond; respectively, to $\sO_{\bP^1}(1)$ and $\sO_{\bP^1}(-\varepsilon)$ under the isomorphism $C_0 \cong \bP^1$.\footnote{Indeed, the latter is formal as $\sO_{\bP(\sF)}(1)=\sO_{X_{\varepsilon}}(C_0)$ and $C_0$ is the section corresponding to $\sF \to \sO_{\bP^1}(-\varepsilon) \to 0$. The former then follows from noticing that $\sO_{X_{\varepsilon}}(C_1) \cong \sO_{X_{\varepsilon}}(C_0 + \varepsilon f)$ pulls back to $\sO_{\bP^1}$, which means that $\sO_{X_{\varepsilon}}(\varepsilon f)$ pulls back to $\sO_{\bP^1}(\varepsilon )$ and then we just divide by $\varepsilon$.} Further,
\[
(F^e_* \sO_{X_{\varepsilon}})\big|_{C_0} \cong \sO_{\bP^1}^{\oplus (1+ \sigma_{\varepsilon})} \oplus \sO_{\bP^1}(-1)^{\oplus (q-1 + \sigma_{\varepsilon+1})} \oplus \bigoplus_{i=1}^{\varepsilon -1} \sO_{\bP^1}(i)^{\oplus \sigma_{\varepsilon -1}}. 
\]
Therefore, $\sE_{e,X_{\varepsilon}}$ is not ample.
\begin{remark} \label{rem.ApplicationLocalAlgebraConesRationalNormalCurve}
Let us point out an interesting application to local algebra. Recall that $X_{\varepsilon}$ is the blowup at the vertex singularity of the projective cone over the rational normal curve in $\bP^{\varepsilon}$; say $0\in P$, with $C_0$ being the exceptional divisor. In particular, restricting $f$ to $X_{\varepsilon} \smallsetminus C_0 = P \smallsetminus \{0\}\subset P$ and pushing it forward to $P$ gives us \emph{a ruling of $P$}; say $L$. Thus, $L$ is a Weil divisor on $P$ of Cartier index $\varepsilon$. Our computations above (with $q \geq \varepsilon$) then show:
\[
F^e_* \sO_P \cong \sO_P^{\oplus (1+\sigma_{\varepsilon})} \oplus \sO_P(-L)^{\oplus (q-1 + \sigma_1 + \sigma_{\varepsilon +1})} \oplus \bigoplus_{i=2}^{\varepsilon -1} \sO_P(-iL)^{\oplus \sigma_i}
\]
where it is worth noting that
\[
1+\sigma_{\varepsilon} = 
\begin{cases} 
q^2/\varepsilon, & \text{if } k=0,\\
\frac{q^2 - \frac{1}{2}\big( \rho_{k,2}^2-2 \rho_{k,1}^2 + (\varepsilon +2 )(2 \rho_{k,1} - \rho_{k,2})\big) + \varepsilon}{\varepsilon},  & \text{otherwise},
\end{cases}
\]
as $\rho_{k,\varepsilon -i} = \varepsilon - \rho_{k,i}$ if $k\neq 0$ and $i=1, \ldots, \varepsilon -1$. Similarly, 
\[
q-1+\sigma_1+\sigma_{\varepsilon+1} = 
\begin{cases} 
q^2/\varepsilon, & \text{if } k=0,\\
\frac{q^2 + \rho_{k,1}(\varepsilon - \rho_{k,1}) - \varepsilon}{\varepsilon},  & \text{otherwise}.
\end{cases}
\]

In particular, localizing at $0$ yields that the $F$-splitting numbers of $\sO_{P,0}$ are $1+\sigma_{\varepsilon}$ (for $q \geq \varepsilon$) as well as a complete description of the $\sO_{P,0}$-module $F^e_* \sO_{P,0}$. For instance, we recover that the $F$-signature of $\sO_{P,0}$ is $1/\varepsilon$. The authors were unaware of such a complete description. We hope that the reader will appreciate the novelty in our rather simple geometric approach.
\end{remark}
\subsection{Blowups of projective spaces along linear subspaces}
\label{sec.BlowupProjSpace}

Let $X \coloneqq \bP^d$, $Y \subset X$ be a linear subspace of $X$ of dimension $r-1$, and $\Tilde{X} \to X$ be the blowup of $X$ along $Y$. Let us assume $d-(r-1) \geq 2$. Recall that $\Tilde{X} $ can be realized as a projective bundle over $\bP^{d-r}$. Indeed, $\tilde{X} \cong \bP(\sF)$ where $\mathcal{\sF} = \sO_{\bP^{d-r}}(1)  \oplus \sO_{\bP^{d-r}}^{\oplus r}$; see \cite[\S 9.3.2]{EisenbudHarris3264AndAllThat}, \cf \cite[V, Example 2.11.4]{Hartshorne}. Moreover, under such an isomorphism, the blowup morphism $\tilde{X} \to X = \bP^d$ is realized by the complete linear system $|\sO_{\bP(\sF)}(1)|$. Let $H, H'$ be the divisors on $\tilde{X}$ defined by the pullback of the hyperplane sections of $X=\bP^d$ and $\bP^{d-r}$; respectively. Then, $\Cl \tilde{X} = \bZ \cdot H  \oplus \bZ \cdot H'$ and $\sO_{\bP(\sF)}(1)$ corresponds to $1\cdot H$. Applying \autoref{pro.PushforwardProjectiveBundles} yields:
\[
F^e_* \sO_{\tilde{X}} \cong \bigoplus_{i=0}^r \sO_{\tilde{X}}(-iH) \otimes \bigoplus_{j=0}^{q-1} \pi^* F^e_* \sO_{\bP^{d-r}}(j)^{\oplus a(i,-j;r-1,e)} \eqqcolon \bigoplus_{i=0}^r \sG_i,
\]
where $\pi \: \tilde{ X} \to \bP^{d-r}$ is the $\bP^r$-bundle morphism and $\sG_i$ are defined in the obvious way. We note that
\[
\sG_0 = \pi^* F^e_* \sO_{\bP^{d-r}} \cong \bigoplus_{k=0}^{d-r}\sO_{\tilde{X}}(-kH')^{\oplus a(k,0;d-r,e)}
\]
and set $b_{0,k} \coloneqq a(k,0;d-r,e)$. For $1 \leq i \leq r-1$, we have:
\begin{align*}
    \sG_i &= \sO_{\tilde{X}}(-iH) \otimes \Bigg(
    \pi^* F^e_* \sO_{\bP^{d-r}}^{\oplus a(i,0;r-1,e)} \oplus \bigoplus_{j=1}^{q-1} \pi^*F^e_* \sO_{\bP^{d-r}}(j)^{\oplus a(i-1, q-j; r-1,e)}
    \Bigg) \\
    &\cong\sO_{\tilde{X}}(-iH) \otimes \Bigg( \bigoplus_{k=0}^{d-r}\sO_{\tilde{X}}(-kH')^{\oplus a(k,0;d-r,e) \cdot a(i,0;r-1,e)} \\
     & \phantom{=\sO_{\tilde{X}}(-iH) \otimes \Bigg(} \oplus \bigoplus_{j=1}^{q-1} \bigoplus_{k=0}^{d-r} \sO_{\tilde{X}}(-kH')^{\oplus a(k,j;d-r,e) \cdot a(i-1, q-j; r-1,e)} \Bigg)\\
     &= \bigoplus_{k=0}^{d-r} \sO_{\tilde{X}}(-iH-kH')^{\bigoplus b_{i,k}}
\end{align*}
where
\[
b_{i,k} \coloneqq a(k,0;d-r,e) \cdot a(i,0;r-1,e) + \sum_{j=1}^{q-1} a(k,j;d-r,e) \cdot a(i-1,q-j;r-1,e).
\]
Likewise,
\begin{align*}
\sG_r &= \sO_{\tilde{X}}(-rH) \otimes \bigoplus_{j=1}^{q-1}\pi^* F^e_* \sO_{\bP^{d-r}}(j)^{\oplus a(r-1,q-j;r-1,e)}\\
&={} \sO_{\tilde{X}}(-rH) \otimes \bigoplus_{j=1}^{q-1} \bigoplus_{k=0}^{d-r} \sO_{\tilde{X}}(-kH')^{\oplus a(k,j;d-r,e)\cdot a(r-1,q-j;r-1,e)}\\
&=\bigoplus_{k=0}^{d-r} \sO_{\tilde{X}}(-rH-kH')^{\oplus \sum_{j=1}^{q-1}a(k,j;d-r,e) \cdot a(r-1,q-j;r-1,e)},
\end{align*}
so we set $b_{r,k} \coloneqq \sum_{j=1}^{q-1}a(k,j;d-r,e) \cdot a(r-1,q-j;r-1,e)$. Summing up:
\[
F^e_* \sO_{\tilde{X}} \cong \bigoplus_{\substack{0 \leq i \leq r \\ 0 \leq k \leq d-r}} \sO_{\tilde{X}}(-iH-kH')^{\oplus b_{i,k}} 
\]
where
\[
b_{i,k} = a(k,0;d-r,e) \cdot a(i,0;r-1,e) + \sum_{j=1}^{q-1} a(k,j;d-r,e) \cdot a(i-1,q-j;r-1,e),
\]
for all $0 \leq i \leq r$, $0 \leq k \leq d-r$.

Let $E$ be the exceptional divisor of $\tilde{X} \to X$ and note that $\Cl \tilde{X} = \bZ \cdot H \oplus \bZ \cdot E$. Set $\sO(a,b) \coloneqq \sO_{\tilde{X}}(aH+bE)$. Observe that $H \sim H'+E$; see \cite[Corollary 9.12]{EisenbudHarris3264AndAllThat}. Hence,
\begin{align*}
F^e_* \sO_{\tilde{X}} \cong \bigoplus_{\substack{0 \leq i \leq r \\ 0 \leq k \leq d-r}} \sO_{\tilde{X}}\big(-iH-k(H-E)\big)^{\oplus b_{i,k}} &= \bigoplus_{\substack{0 \leq i \leq r \\ 0 \leq k \leq d-r}}
\sO_{\tilde{X}}\big(-(i+k)H+kE)\big)^{\oplus b_{i,k}}   \\
&= \bigoplus_{\substack{0 \leq i \leq r \\ 0 \leq k \leq d-r}}
\sO(-i-k, k)^{\oplus b_{i,k}} 
\end{align*}

It is noteworthy that setting $d=2, r=1$ recovers our computation for $X_1$ in \autoref{ex.HirzebruchSurfaces} as $\sO_{X_1}(C_0)=\O(0,1)$ and $\sO_{X_1}(C_1)=\O(1,0)$. 

Pulling $F^e_* \sO_{\tilde{X}}$ back to the big open $\tilde{X} \smallsetminus E = X \smallsetminus Y \subset X$ and then pushing it forward to $X=\bP^d$ yields $\sum_{i+k=l}b_{i,k} = a(l,0;d,e)$ for all $l=0,\ldots, d$ (independently of $r$). For $r=1$: 
\[
a(l,0;d,e) = b_{0,l} + b_{1,l-1} = a(l,0;d-1,e) + \sum_{j=1}^{q-1}a(l-1,j;d-1,e)
\]
for all $l=0,\ldots, d$. In other words,
\[
\sum_{j=1}^{q-1} a(l-1,j;d-1,e) = a(l,0;d,e)-a(l,0;d-1,e).
\]
Adding $a(l-1,0;d-1,e)$ on both sides yields:
\begin{equation} \label{eqn.TRickySUm}
    \sum_{j=0}^{q-1}a(l-1,j;d-1,e) = a(l,0;d,e)-a(l,0;d-1,e) + a(l-1,0;d-1,e),
\end{equation}
for all $d,e$ and $l=1, \ldots, d$.

Next, recall that the exceptional divisor $E \to Y$ is realized as the projective bundle $\bP(\sI/\sI^2) \to Y \cong \bP^{r-1}$ where $\sI$ is the ideal sheaf cutting out $Y$ and the conormal bundle of $E \subset \tilde{X}$ corresponds to $\sO_E(-1) \coloneqq \sO_{\bP(\sI/\sI^2)}(-1)$. Therefore, the pullback of $F^e_*\sO_{\tilde{X}}$ to $E$ is
\[
\big(F^e_* \sO_{\tilde{X}}\big)\Big|_E \cong \bigoplus_{0 \leq k \leq d-r} \sO_{E}(-k)^{\oplus q^{r-1}(a(k+1,0;d-(r-1),e)-a(k+1,0;d-r,e)+a(k,0;d-r,e))} 
\]
as the pullback of $\sO(-i-k,k)$ to $E$ is $\sO_{E}(-k)$ and 
\begin{align*}
\sum_{i=0}^r b_{i,k} &= q^{r-1} \cdot \sum_{j=0}^{q-1}a(k,j;d-r,e)\\
&= q^{r-1}\big(a(k+1,0;d-(r-1),e)-a(k+1,0;d-r,e)+a(k,0;d-r,e)\big),
\end{align*}
where the last equality is an application of \autoref{eqn.TRickySUm}. Further, by setting $k=0$ above and after a short calculation, we see that 
\[
\sO_{E}^{\oplus q^r \binom{q+d-r}{d-r}}
\] is a direct summand of $\big(F^e_* \sO_{\tilde{X}}\big)\big|_E$ and so $\sE_{e,\tilde{X}}$ is not ample. Consequently:

\begin{proposition}\label{lem.FrobPushForwardLocalgeneralBlowups}
Let $S$ be a smooth variety of dimension $d$ and $f: X \to S$ be the blowup of $S$ along a smooth closed subvariety $C \subset S$ of dimension $c-1\leq d-2$. Further, set an isomorphism $g \: \hat{\bA}^d \cong \Spec \hat{\sO}_{S,s} \to S$ where $s \in S$ is a smooth closed point contained in $C$ and write $\hat{\bA}^d = \Spec \kay \llbracket x_1,\ldots ,x_d \rrbracket$ with $(x_{c},\ldots,x_d)$ being local equations for $C$. Consider the following cartesian diagram
\[
\xymatrix@C=5em{
\bP_{\hat{\bA}^{c-1}}^{d-c} \ar[r] \ar[d] & \hat{X} \ar[r]^-{h} \ar[d]_-{\hat{f}} & X \ar[d]^-{f} \\
\hat{\bA}^{c-1} \ar[r]^-{x_c,\ldots,x_d=0} & \hat{\bA}^d \ar[r]^-{g} & S
}
\]
so that $\hat{f}$ is the blowup of $\hat{\bA}^d$ with respect to the ideal $(x_{c}, \ldots, x_d)$. Then,
\[
h^* F^e_* \sO_X = F^e_* \sO_{\hat{X}} \cong \bigoplus_{0 \leq k \leq d-c} \sO_{\hat{X}}(kE)^{\oplus q^{c-1}(a(k+1,0;d-(c-1),e)-a(k+1,0;d-c,e)+a(k,0;d-c,e))} 
\]
where $E$ is the exceptional divisor of $\hat{f}$. Furthermore, the pullback of $F^e_* \sO_X$ to $E \cong \bP_{\hat{\bA}^{c-1}}^{d-c}$ is
\[
\big(F^e_* \sO_{X}\big)\big|_E \cong \bigoplus_{0 \leq k \leq d-c} \sO_E(-k)^{\oplus q^{c-1}(a(k+1,0;d-(c-1),e)-a(k+1,0;d-c,e)+a(k,0;d-c,e))} ,
\]
which has $\sO_{E}^{\oplus q^c \binom{q+d-c}{d-c}}$ as a direct summand. Therefore, $\sE_{e,X}$ is not ample.
\end{proposition}
The importance of \autoref{lem.FrobPushForwardLocalgeneralBlowups} for us can already be appreciated:
\begin{corollary} \label{cor.(-1)CurvesRuleOut}
If $X$ is a smooth surface that admits a $(-1)$ curve, then $\sE_{e,X}$ is not ample.
\end{corollary}
\begin{proof}
Use Castelnuovo's contraction theorem and \autoref{lem.FrobPushForwardLocalgeneralBlowups}.
\end{proof}

\begin{remark} \label{rem.DirecCumbersoneStrategy}
Let us stress how \autoref{cor.(-1)CurvesRuleOut} works. We need to show that $\sE_{e,X}|_C$ is not ample where $\bP^1 \cong C \subset X$ is the $(-1)$-curve. By Castelnuovo's contraction theorem, $C$ is the exceptional divisor of a blowup $X \to S$ at a closed point $s \in S$ of some smooth surface $S$. In principle, the computation of the restriction of $\sE_{e,X}|_C$ can be carried out locally around $s \in S$ yet we do something quite different. Since the computation is local, we are free to replace $X$ by another surface which is isomorphic to $X$ around $C$, say the blowup of $\bP^2$ at the origin. Then, we exploit the global geometry of such blow up to carry out the intersection computation of interest globally. This kind of idea will be exported to the threefold case in \autoref{Main.SubsectionFibrationsContractions} in the proof of \autoref{prop.RulingOutCases}. Nonetheless, we think it is instructive to show how the direct local computation works in the simplest case of $\Bl_0 \bP^2$. Set $x,y$ to be local coordinates around $0 \in \bP^2$. We illustrate next how to describe $(F^e_* X)|_{\bP^1}$ locally, where $X \coloneqq \Bl_0 \bA^2$ and $\bP^1 \subset X$ is the exceptional divisor. Recall that $X$ is described by the affine charts $\kay[x,y/x]$ and $\kay[x/y,y]$ inside $\kay(x,y)$. On $\kay[x,y/x]$, $F^e_* \sO_X$ admits the decomposition $\bigoplus_{0 \leq i,j \leq q-1} \kay[x,y/x] F^e_* x^i (y/x)^j$ whereas $F^e_* \sO_X$ equals $\bigoplus_{0 \leq i,j \leq q-1} \kay[x/y,y] F^e_* (x/y)^i y^j$ on the chart $\kay[x/y,y]$. Thus, $(F^e_* \sO_X)|_{\bP^1}$ equals $\bigoplus_{0 \leq i,j \leq q-1} \kay[y/x] F^e_* x^i (y/x)^j$ on $\kay[y/x]$ and likewise $F^e_* \sO_X$ is $\bigoplus_{0 \leq i,j \leq q-1} \kay[x/y] F^e_* (x/y)^i y^j$ on $\kay[x/y]$; where $\bP^1$ is being realized by the affine charts $\kay[y/x]$ and $\kay[x/y]$ inside $\kay(y/x)$. Now, observe that
\[
F^e_* x^i (y/x)^j = \begin{cases} 
F^e_* (x/y)^{i-j}y^i, & \text{if } i \geq j,\\
(y/x)F^e_* (x/y)^{q-(j-i)}y^i,  & \text{if } j>i.
\end{cases}
\]
In particular,
\[
\bigoplus_{0 \leq i,j \leq q-1} \kay[y/x] F^e_* x^i (y/x)^j = \bigoplus_{i=0}^{q-1} \left( \bigoplus_{j=0}^i \kay[y/x]F_*^e(x/y)^j y^i \oplus \bigoplus_{j=i+1}^{q-1} \kay[y/x](y/x)F_*^e(x/y)^j y^i\right).
\]
Hence, by gluing $\kay[y/x]F_*^e(x/y)^j y^i$ with $\kay[x/y]F_*^e(x/y)^j y^i$ for $j \leq i$; obtaining a copy of $\sO_{\bP^1}$, and $\kay[y/x](y/x)F_*^e(x/y)^j y^i$ with $\kay[x/y]F_*^e(x/y)^j y^i$ for $j>i$; obtaining a copy of $\sO_{\bP^1}(-1)$, we see that $(F^e_* \sO_X)|_{\bP^1}$ is a direct sum of $1+\cdots +q=q(q+1)/2$ many copies of $\sO_{\bP^1}$ and $q^2-q(q+1)/2=q(q-1)/2$ many copies of $\sO_{\bP^1}(-1)$; agreeing with our previous computations. 
\end{remark}

\subsection{Cones} \label{sec.CONESEXAMPLES}
To set notation, we recall some general constructions. For details, see \cite[\href{https://stacks.math.columbia.edu/tag/0EKF}{Tag 0EKF}]{stacks-project}, \cite[II, Exercise 6.3 and V, Example 2.11.4]{Hartshorne}. Let $S=\bigoplus_{i \in \bN} S_i$ be a graded ring that is finitely generated by $S_1$ as an $S_0$-algebra and suppose $S_0$ and so $S$ to be noetherian (e.g. $S_0=\kay$). That is, $S$ is a standard graded ring. Set:
\[
Z \coloneqq \Spec S_0,\quad V \coloneqq \Proj S, \quad  C\coloneqq \Spec S, \quad P \coloneqq \Proj S[t]
\] 
where $S[t]$ is the graded ring obtained from $S$ by adding a free variable $t$ in degree $1$. These are all quasi-compact $Z$-schemes. Additionally, let us denote by $\sigma \:L \to V$ the cone defined by $\sO_V(1)$, that is, $L \coloneqq \Spec_V \bigoplus_{n\geq 0} \sO_V(n) = \Spec_V \Sym \sO_V(1)$, which is a line bundle as $S$ is generated by $S_1$ as an $S_0$-algebra. Then, we have a commutative diagram
\begin{equation}\label{eqn.CommDiagnAffinCones}
\xymatrix{
V \ar[r] \ar[d] &L \ar[r]^-{\sigma} \ar[d]^-{g} & V \ar[d] \\
Z \ar[r] &C \ar[r] & Z
}
\end{equation}
where:
\begin{itemize}
\item $V \to L$ is the zero section of $\sigma$; which is defined by $0 \in H^0\big(V,\sO_V(-1)\big)$ as $\sO_V(1)$ is an invertible sheaf on $V$, \item  $Z \to C$ is the closed embedding cut out by the irrelevant ideal $S_{+} \subset S$, and
\item $g$ is the canonical morphism $g \: L \to \Spec H^0(L,\sO_L) \to C$.
\end{itemize}
Since $S$ is generated by $S_1$ as an $S_0$-algebra, $g$ is the blowup of $C$ along $Z$ and the section $V \to L$ is its exceptional divisor.  

Next, let $\pi \: X \to V$ denote the $\bP^1$-bundle $\bP(\sF) \to V$ defined by $\sF= \sO_{V} \oplus \sO_V(1)$. Note that there are isomorphisms of graded $\sO_V$-algebras
\[
\Sym \sF \cong \Sym \sO_V(1) \otimes_{\sO_V}  \Sym \sO_V  \cong \big(\Sym \sO_V(1) \big) \otimes_{\sO_V} \sO_V[t] \eqqcolon \big(\Sym \sO_V(1) \big)[t]
\]
In particular, the closed subscheme of $X$ defined by $t=0$ is the section of $\pi \: X \to V$ defined by the direct summand quotient $\sF \to \sO_V(1) \to 0$, whose corresponding Cartier divisor we denote by $H$. Further, the open complement of $H \subset X$ is $L = \Spec_V \Sym \sO_V(1)$ and $L \subset X \xrightarrow{\pi} V$ coincides with $\sigma \: L \to V$.

Let $E$ be the Cartier divisor on $X$ defined by the section of $\pi$ corresponding to the other direct summand quotient $\sF \to \sO_V \to 0$. Since $\sF$ splits as a direct sum, $H \cap E = \emptyset$. Moreover, the restriction of $E$ to the open $L$ is none other than the zero section of $\sigma\: L \to V$.

Note that $\Gamma_*(X, \sO_{\bP(\sF)}(1)) = H^0(L,\sO_L)[t]$ as graded rings. Thus, the canonical graded homomorphism $S[t] \to \Gamma_*(X, \sO_{\bP(\sF)}(1)) $ defines a morphism $f\: X \to P$ \cite[\href{https://stacks.math.columbia.edu/tag/01NA}{Tag 01NA}]{stacks-project}, which restricts to $g\: L \to C$. Thus, there is a commutative diagram that extends \autoref{eqn.CommDiagnAffinCones}:
\[
\xymatrix{
V \ar[r] \ar[d] &X \ar[r]^-{\pi} \ar[d]^-{f} & V \ar[d] \\
Z \ar[r] &P \ar[r] & Z
}
\]
In particular, $f$ is the blowup of $P$ along $Z$ and the section $V \to X$; which defines $E$, is its exceptional divisor. Moreover, $f^* \sO_P(1) = \sO_{\bP(\sF)}(1)$. 

Let $G$ be a Cartier divisor on $X$ such that $\sO_X(G)=\pi^* \sO_V(1)$. Then, the tautological quotient $\pi^* \sF \to \sO_{\bP(\sF)}(1) \to 0$ is $\sO_X \oplus \sO_X (G) \to \sO_{\bP(\sF)}(1)$. We see that the divisor of zeros of the global section of $\sO_{\bP(\sF)}(1)$ defined via the tautological quotient is precisely $H$, whence $\sO_{\bP(\sF)}(1) \cong \sO_X(H)$. Thus, the tautological quotient can be thought of as $\sO_X \oplus  \sO_X(G) \to \sO_X(H) \to 0$. Twisting it by $\sO_X(-G)$, we obtain a global section of $\sO_X(H-G)$ whose divisor of zeros is $E$. In other words, we obtain the relation $H-G \sim E$.

Recall that $\pi^* \: \Pic V \to \Pic X$ and $\bZ \to \Pic X$; $1 \mapsto \sO_X(H)$, define an isomorphism $\bZ \oplus \Pic V \xrightarrow{\cong} \Pic X$. Let $\sN$ be an invertible sheaf on $V$, then:
\begin{align} \label{eqn.GenralConeFormula}
F^e_* \big(\sO_X(nH) \otimes \sN\big) 
\cong {} &\sO_X\big(\lfloor n/q \rfloor H\big) \otimes \bigoplus_{j=0}^{[n]_q} \pi^* F^e_* \big(\sO_V(j) \otimes \sN\big)\\ \nonumber
&\oplus \sO_X\big( (\lfloor n/q \rfloor - 1) H\big) \otimes \bigoplus_{j=[n]_q+1}^{q-1} \pi^* F^e_* \big(\sO_V(j) \otimes \sN\big)
\end{align}
It is difficult to say much more for such a general $V=\Proj S$. In what follows, we specialize to Veronese and Segre embeddings. 

\subsubsection{Veronese embeddings} \label{sec.VeroneseEmbeddings}
Let $S$  be the $\varepsilon$-th Veronese subring of the standard graded polynomial ring $\kay[x_0,\ldots, x_d]$. That is, $V$ is the $\varepsilon$-th Veronese embedding of $\bP^d$, and $C$ and $P$ are; respectively, the affine and projective cones over $V$. Thus, $X$ is the blowup of $P$ at its vertex and it can be realized as the $\bP^1$-bundle over $\bP^d$ defined by $\sF= \sO_{\bP^d} \oplus \sO_{\bP^d} (\varepsilon)$. Specializing to $d=1$ recovers the examples in \autoref{ex.HirzebruchSurfaces}. Denote by $H' \subset X$ the pullback of a hyperplane along the morphism $\pi \: X \to \bP^d$. In particular, $G = \varepsilon H'$, and so $H \sim E + \varepsilon H'$. The above formula \autoref{eqn.GenralConeFormula} becomes:
\begin{align*}
F^e_* \big(\sO_X(nH+n'H')\big) 
\cong {} &\sO_X\big(\lfloor n/q \rfloor H\big) \otimes \bigoplus_{j=0}^{[n]_q} \pi^* F^e_* \big(\sO_{\bP^d}(\varepsilon j+n')\big)\\
&\oplus \sO_X\big( (\lfloor n/q \rfloor - 1) H\big) \otimes \bigoplus_{j=[n]_q+1}^{q-1} \pi^* F^e_* \big(\sO_{\bP^d}(\varepsilon j+n') \big),
\end{align*}
where
\[
\pi^* F^e_* \big(\sO_{\bP^d}(j+n') \big) \cong \bigoplus_{l=0}^d \sO_X\big((\lfloor (\varepsilon j+n')/q \rfloor - l) H'\big)^{\oplus a(l,[\varepsilon j+n']_q;d,e)}.
\]
By the projection formula, it suffices to focus on the case $0 \leq n,n' \leq q-1$. Hence,
\begin{align*}
    F^e_* \big(\sO_X(nH+n'H')\big) 
\cong {} & \bigoplus_{l=0}^d \bigoplus_{j=0}^{n}\sO_X\big((\lfloor (\varepsilon j+n')/q \rfloor - l) H'\big)^{\oplus a(l,[\varepsilon j+n']_q;d,e)} \\
&\oplus \bigoplus_{l=0}^d \bigoplus_{j=n+1}^{q-1}\sO_X\big(-E+(\lfloor (\varepsilon j+n')/q \rfloor -\varepsilon - l) H'\big)^{\oplus a(l,[\varepsilon j+n']_q;d,e)} \\
\cong {} & \bigoplus_{l=0}^d \bigoplus_{j=0}^{n}\sO_X\big((\lfloor (\varepsilon j+n')/q \rfloor - l) H'\big)^{\oplus a(l,[\varepsilon j+n']_q;d,e)} \\
&\oplus \bigoplus_{l=0}^d \bigoplus_{j=1}^{q-1-n}\sO_X\big(-E+(\lfloor (-\varepsilon j+n')/q \rfloor - l) H'\big)^{\oplus a(l,[-\varepsilon j+n']_q;d,e)},
\end{align*}
where the last equality follows from noting that $\lfloor (\varepsilon(q-j)+n')/q\rfloor   -\varepsilon = \lfloor (-\varepsilon j + n')/q \rfloor $ and $[\varepsilon(q-j)+n']_q = [-\varepsilon j +n']_q$ for all $j=1,\ldots ,q-1$. At this point, the computation becomes quite involved. We next illustrate the easier, yet more important cases. \emph{Our first simplification is the following assumption:}
\[
q \geq \varepsilon - n' \geq 1.
\]
Next, we introduce two partitions of $\{1,\ldots, q-1\}$.
\[
\{0,\ldots, q-1\} = I_1 \cup I_1 \cup \cdots \cup I_{\varepsilon -1} \cup I_{\varepsilon},
\]
where for $i=1,\ldots, \varepsilon-1$ :
\begin{align*}
I_i &\coloneqq \big\{ \lfloor ((i-1)q-1-n')/\varepsilon \rfloor +1, \ldots, \lfloor  (iq-1-n')/\varepsilon \rfloor \big\},\\
I_{\varepsilon} & \coloneqq \big\{ \lfloor ((\varepsilon-1)q-1-n')/\varepsilon \rfloor +1, \ldots, q-1 \big\}.
\end{align*}
Thus, if $j \in I_i$ then $\lfloor (\varepsilon j + n')/q \rfloor = i-1$ and $[\varepsilon j + n']_q = \varepsilon j + n' -(i-1)q$. The other partition is:
\[
\{1,\ldots, q-1\} = J_1 \cup J_2 \cup \cdots \cup J_{\varepsilon -1} \cup J_{\varepsilon}
\]
where for $i=1,\ldots, \varepsilon-1$ :
\begin{align*}
J_i &\coloneqq \big\{ \lfloor ((i-1)q+n')/\varepsilon \rfloor +1, \ldots, \lfloor  (iq+n')/\varepsilon \rfloor \big\},\\
J_{\varepsilon} & \coloneqq \big\{ \lfloor ((\varepsilon-1)q+n')/\varepsilon \rfloor +1, \ldots, q-1 \big\}.
\end{align*}
Hence, if $j \in J_i$ then $\lfloor (-\varepsilon j + n')/q \rfloor = -i$ and $[-\varepsilon j + n']_q = iq -\varepsilon j + n'$. Of course, if $n'=0$, this is the partition we had in \autoref{ex.HirzebruchSurfaces}. It is convenient to define $J_{-1}\coloneqq \{0\}$.

Let $i_n, i_n'$ be defined by $n \in I_{i_n}$ and $q-1-n \in J_{i_n'}$. Then,
\[
F^e_* \sO_X(nH + n'H') \cong \bigoplus_{k=-i_n+1}^d
 \sO_X(-kH')^{\oplus \varsigma_k } \oplus \bigoplus_{k=1}^{i_n'+d} \sO_X(-E -kH')^{\oplus \sigma_k},
 \]
where $\varsigma_k$ and $\sigma_k$ are computed as follows. For each $l=0,\ldots, d$ and $i=1,\ldots, \varepsilon$, define:
\begin{align*}
\varsigma_{i}^{(l)} &\coloneqq \sum_{j \in I_i \cap [0,n]}  a(l,[\varepsilon j + n']_q;d,e) =  \sum_{j \in I_i \cap [0,n]} a(l,\varepsilon j + n' -(i-1)q;d,e) \\
\sigma_i^{(l)}  &\coloneqq \sum_{j \in J_i \cap [1, q-1-n]} a(l,[-\varepsilon j +n']_q;d,e) = \sum_{j \in J_i \cap [1, q-1-n]} a(l,iq-\varepsilon j +n';d,e).
\end{align*}
Then,
\[
\varsigma_k = \sum_{i-1-l = -k} \varsigma_i^{(l)} = \sum_{l=k+i-1} \varsigma_i^{(l)} = \sum_{i=1}^{i_n} \varsigma_{i}^{(k+i-1)},  
\]
and likewise
\[
\sigma_k = \sum_{i+l=k} \sigma_{i}^{(l)} = \sum_{i=1}^{i_n'}\sigma_i^{(k-i)}.
\]
To go on, we set $n,n'=0$.  Then, $i_n=1$, $i'_n= \varepsilon$, and:
\[
F^e_* \sO_X = \bigoplus_{k=0}^d \sO_X(-kH')^{\oplus \varsigma_k} \oplus \bigoplus_{k=1}^{\varepsilon + d} \sO_X(-E - k H')^{\oplus \sigma_k}, 
 \]
where,
\[
\varsigma_k = \varsigma_1^{(k)} = a(k,0;d,e), \quad \sigma_k = \sum_{\substack{1 \leq i \leq \varepsilon \\ 
0 \leq l \leq d \\
i+l = k} } \sigma_i^{(l)} = \sum_{\substack{1 \leq i \leq \varepsilon \\ 
0 \leq l \leq d \\
i+l = k} } \sum_{j \in J_i} a(l,iq-\varepsilon j;d,e),
\]
and
\begin{align*}
J_i &= \big\{ ((i-1)q - \rho_{c,(i-1)})/\varepsilon + 1, \ldots, (iq-\rho_{c,i})/\varepsilon \big\}, \quad i=1,\ldots, \varepsilon.
\end{align*}
where $\rho_{c,i}$ is defined as in \autoref{ex.HirzebruchSurfaces} and $c$ is the residue of $q$ modulo $\varepsilon$.

Recall that $E \cong \bP^d$ and, under this isomorphism, $\sO_X(H')\big|_E $ and $\sO_X(E)\big|_E$ correspond to $\sO_{\bP^d}(1)$ and $\sO_{\bP^d}(-\varepsilon)$. Then,
\[
(F^e_* \sO_X)\big|_{E \cong \bP^d} \cong \bigoplus_{k=0}^{d}\sO_{\bP^d}(-k)^{\oplus(\varsigma_k + \sigma_{\varepsilon + k})} \oplus \bigoplus_{k=1}^{\varepsilon -1 } \sO_{\bP^d}(k)^{\oplus \sigma_{\varepsilon-k}}. 
\]
Therefore, by looking at the summand $k=0$, we find that $\sE_{e,X}$ is not ample.
\begin{remark} \label{rem.ApplicationLocalAlgebraVeronese}
For Hirzebruch surfaces, we can use the above to compute $F^e_* \sO_P$ even if $P$ is singular; see \autoref{rem.ApplicationLocalAlgebraConesRationalNormalCurve}. Indeed, let $L \subset P$ be the pushforward to $P$ of the restriction of $H'$ to $X \smallsetminus E = P \smallsetminus \{0\} \subset P$. Thus, $L$ is a $\bQ$-Cartier divisor on $P$ with Cartier index $\varepsilon$, indeed $\varepsilon L \sim H$ where $H \subset P$ denotes the closure of the restriction of $H \subset X$ to $X \smallsetminus E = P \smallsetminus \{0\}$. Restricting $F^e_* \sO_X$ to $X \smallsetminus E = P \smallsetminus \{0\} \subset P$ and then pushing it forward to $P$ yields:
\[
F^e_* \sO_P \cong \bigoplus_{k=0}^d \sO_P(-k L)^{\oplus (\varsigma_k  + \sigma_{k+\varepsilon})} \oplus \bigoplus_{k=1}^{\varepsilon -1} \sO_P (-kL)^{\oplus \sigma_k}. 
\]
In particular, if $d \geq \varepsilon -1$ then $\varsigma_0 + \sigma_{\varepsilon} = 1+ \sigma_{\varepsilon}$ is the $e$-th $F$-splitting number of $S$. In general, it is $\sum_{k=0}^{\lfloor d/\varepsilon \rfloor} \varsigma_{k\varepsilon} + \sigma_{(k+1)\varepsilon}$. To the best of the author's knowledge, such explicit description of $F^e_* \sO_P$ and so of $F^e_* \sO_{P,0}$ has not been worked out before. We recover that $s(\sO_{P,0}) = 1/\varepsilon$.
\end{remark}
\subsubsection{Segre embeddings} \label{sec.SegreEmbeddings}
Let $S \coloneqq \kay [x_0,\ldots ,x_r] \# \kay[y_0,\ldots, y_s]$ be the Segre product of two standard graded polynomial $\kay$-algebras. Then, $V \cong \bP^r \times \bP^s$ and $X$ is the $\bP^1$-bundle over $\bP^r \times \bP^s$ defined by $\sF = \sO_{\bP^r \times \bP^s} \oplus \sO_{\bP^r \times \bP^s}(1,1)$. We let $\sO_X(G_1) = \pi^* \sO(1,0)$ and $\sO_X(G_2) = \pi^* \sO(0,1)$, so that $G=G_1+G_2$. Thus, the Cartier divisors $H, G_1, G_2$ are free generators of $\Pic X$. Letting $0 \leq n,n_1, n_2 \leq q-1$:  
\[
F^e_* \sO_X(n H+n_1 G_1 + n_2 G_2) \cong \bigoplus_{j=0}^n \pi^* F^e_* \sO(j+n_1, j +n_2) \oplus  \bigoplus_{j=n+1}^{q-1}  \sO_X(-H) \otimes \pi^* F^e_*\sO(j+n_1, j+n_2),
\]
which is isomorphic to the direct sum of
\[
\bigoplus_{j=0}^n \bigoplus_{\substack{0 \leq k \leq r\\ 0 
 \leq l \leq s}} \sO_X((\lfloor (j+n_1)/q \rfloor -k)G_1 + (\lfloor (j+n_2)/q \rfloor-l) G_2)^{\oplus a(k,[j+n_1]_q;r,e)a(l,[j+n_2]_q;s,e)}
\]
with
\[
\bigoplus_{j=n+1}^{q-1} \bigoplus_{\substack{0 \leq k \leq r\\ 0 
 \leq l \leq s}} \sO_X(-H+(\lfloor (j+n_1)/q \rfloor - k) G_1 + (\lfloor (j+n_2)/q \rfloor-l )G_2)^{\oplus a(k,[j+n_1]_q;r,e)a(l,[j+n_2]_q;s,e)}.
\]

Let us set $n,n_1,n_2=0$. Then,
\begin{align*}
&F^e_* \sO_X \\
\cong{}  &\bigoplus_{\substack{0 \leq k \leq r\\ 0 
 \leq l \leq s}} \sO_X(-kG_1 -l G_2)^{\oplus a(k,0;r,e)a(l,0;s,e)} \oplus \bigoplus_{\substack{0 \leq k \leq r\\ 0 
 \leq l \leq s}} \sO_X(-H-kG_1 -l G_2)^{\oplus \sum_{j=1}^{q-1} a(k,j;r,e)a(l,j;s,e)}\\
 \cong{}  &\bigoplus_{\substack{0 \leq k \leq r\\ 0 
 \leq l \leq s}} \sO_X(-kG_1 -l G_2)^{\oplus a(k,0;r,e)a(l,0;s,e)} \oplus \bigoplus_{\substack{1 \leq k \leq r+1\\ 1 
 \leq l \leq s+1}} \sO_X(-E-kG_1 -lG_2)^{\oplus \sigma_{k-1,l-1}},
\end{align*}
where $\sigma_{k,l} \coloneqq \sum_{j=1}^{q-1} a(k,j;r,e)a(l,j;s,e)$. In this example, $E \cong \bP^r \times \bP^s$, and $\sO_X(E)\big|_E$, $\sO_X(G_1)\big|_E$, $\sO_X(G_2)\big|_E$ correspond to $\sO(-1,-1)$, $\sO(1,0)$, $\sO(0,1)$; respectively. Then,
\[
(F^e_* \sO_X)\big|_{E \cong \bP^r \times \bP^s} \cong  \bigoplus_{\substack{0 \leq k \leq r\\ 0 
 \leq l \leq s}} \sO(-k,-l)^{\oplus \sum_{j=0}^{q-1}a(k,j;r,e)a(l,j;s,e)}.
\]
As before, looking at $k,l=0$ let us conclude that $\sE_{e,X}$ is not ample.
\begin{remark} \label{rem.ApplicationLocalAlgebraSegreProduc}
We may describe $F^e_* \sO_P$ as in \autoref{rem.ApplicationLocalAlgebraConesRationalNormalCurve} and \autoref{rem.ApplicationLocalAlgebraVeronese}. Let $L_i$ be the restriction of $G_i$ to $X \smallsetminus E = P \smallsetminus \{0\}$ followed by its pushforward to $P$. Then, $L_i$ is a Weil divisor on $P$ and $\Cl P = \bZ \cdot L_1 \oplus \bZ \cdot L_2$. Of course, $L_1 + L_2 \sim H$ where $H \subset P$ denotes the restriction of $H$ to $X \smallsetminus E = P \smallsetminus \{0\}$ followed by its pushforward to $P$. Then, $L_1 + L_2 \sim 0$ on the affine cone $P \smallsetminus H = \Spec S$ and
\[
F^e_* \sO_{P\smallsetminus H} \cong \bigoplus_{i=-r}^{s} \sO_{P\smallsetminus H}(iL)^{\oplus \sum_{l-k=i}\sum_{j=0}^{q-1} a(k,j;r,e)a(l,j;s,e)}
\]
where $L$ denotes the class of $L_1$ on $P\smallsetminus H$, which freely generates $\Cl (P \smallsetminus H) = \Cl S$. Looking at $i=0$, the $e$-th $F$-splitting number of $S$ is $\sum_{l=k}\sum_{j=0}^{q-1} a(k,j;r,e)a(l,j;s,e)=\sum_{k}\sum_{j=0}^{q-1} a(k,j;r,e)a(k,j;s,e)$. From the proof of \autoref{pro.PushforwardProjectiveBundles}, we know that $a(k,j;r,e)$ is the coefficient of $u^{j+kq}$ in $(1+u+\cdots + u^{q-1})^{r+1}$ and analogously for $a(l,j;s,e)$. Thus, the above proves that the $e$-th $F$-splitting number of $S$ is the sum of the coefficients of monomials $\{u^kv^k\}_k$ in the product $(1+u+\cdots + u^{q-1})^{r+1} (1+v+\cdots + v^{q-1})^{s+1}$ thereby recovering \cite[Example 7]{SinghFSignatureOfAffineSemigroup}.
\end{remark}
\subsection{Quadrics} \label{ex.QuadricsAnalysis}
Let $\bQ^d$ be the $d$-dimensional smooth quadric. That is, $\bQ^d$ is the hypersurface of $\mathbb{P}^{d+1}$ cut out by the equation $x_0^2 + x_1x_2 + \cdots + x_d x_{d+1}=0$ if $d$ is odd or by the equation $x_0 x_1 + \ldots + x_d x_{d+1}=0$ if $d$ is even. The Frobenius pushforwards of invertible sheaves (in fact, of arithmetically Cohen--Macaulay locally free sheaves) on $\bQ^d$ has been thoroughly described in \cite{AchingerFrobeniusPushForwardQuadrics,LangerDAffinityFrobeniusQuadrics}, where the reader can find the precise description. Here, we are interested on the positivity of $\sE_{e}$, which we study below. From \autoref{ex.ProducProjectiveSpaces}, we know that $\sE_e$ is ample for $d=1$ but not for $d=2$ as $\bQ^1 \cong \bP^1$ and $\bQ^2 \cong \bP^1 \times \bP^1$. We show next that $\sE_e$ is ample for $d \geq 3$ if and only if $p>2$. 

Recall that $\omega_{\bQ^d} = \sO_{\bQ^d}(-d)$ so that 
\[
F^e_* \omega_{\bQ^d}^{1-q} \cong F^e_*\sO_{\bQ^d}(d(q-1)).
\]
Let $\sS$ denote the spinor bundle on $\bQ^d$ (following the notation in \cite[\S 1.2]{AchingerFrobeniusPushForwardQuadrics}),\footnote{That is, $\sS$ is the spinor bundle if $d$ is odd and is the direct sum of the two spinor bundles if $d$ is even.} which is a locally free sheaf of rank $2^{\lfloor d/2 \rfloor}$. See \cite{LangerDAffinityFrobeniusQuadrics,AchingerFrobeniusPushForwardQuadrics,AddingtonSpinorSheavesOnSingularQuadrics,AddingtonSpinorsSheavesCIQuadrics,OttavianiSpinorBundleOnQuadrics,KapranovDerivedCategoryCoherentBdlsQuadircs} for more on spinor sheaves on quadrics.  According \cite[Theorems 2 and 3]{AchingerFrobeniusPushForwardQuadrics}, $F^e_*\omega_{\bQ^d}^{1-q}$ admits a direct sum decomposition
\begin{equation} \label{eqn.QuadricDecomposition}
F^e_*\omega_{\bQ^d}^{1-q} \cong \bigoplus_{i \in \bZ} \sO_{\bQ^d}(i)^{\oplus a_i} \oplus \bigoplus_{j \in \bZ} \sS(j)^{\oplus b_j} 
\end{equation}
where $a_i \neq 0$ if and only if 
\[
0 \leq d(q-1) - iq \leq d(q-1)
\]
and $b_j \neq 0$ if and only if
\[
\begin{cases} 
\frac{d}{2}\big(p^e-p^{e-1}\big)-p^e+p^{e-1} \leq d\big(p^e-1\big) - jp^e \leq \frac{d}{2}\big( p^e-p^{e-1}\big) - p^{e-1} + d\big( p^{e-1}-1\big), &  p \neq 2,\\
\big(\lfloor d/2 \rfloor - 1\big) 2^{e-1} \leq d(2^e-1) - j2^e \leq d(2^e-1)-2^e-\big(\lfloor d/2 \rfloor - 1\big) 2^{e-1} ,  &  p=2.
\end{cases}
\]
Equivalently, $b_j \neq 0$ if and only if
\[
\begin{cases} 
\frac{d}{2}p^e-\left(\frac{d}{2}-1\right)p^{e-1}\leq jp^e \leq  \left(\frac{d}{2}+1\right)p^e+\left(\frac{d}{2}-1\right)p^{e-1} -d, &  p \neq 2,\\
2^e+\big(\lfloor d/2 \rfloor - 1\big) 2^{e-1} \leq j2^e \leq d(2^e-1)-\big(\lfloor d/2 \rfloor - 1\big) 2^{e-1},  &  p=2.
\end{cases}
\]
In particular, $a_i \neq 0$ if and only if $0 \leq i \leq d-d/q$. Thus, $a_i=0$ unless $0 \leq i \leq d-1$. 

To analyze the vanishing of $b_j$, we must consider whether or not $p = 2$. \emph{Suppose $p \neq 2$ first}. Note that $d/2-1 \geq 3/2-1 = 1/2 >0$ (if $d \geq 3$). In particular, if $b_j \neq 0$ then
\[
j \geq \frac{d}{2} - \left(\frac{d}{2}-1\right) \frac{1}{p} \geq \frac{d}{2} - \left(\frac{d}{2}-1\right) \frac{1}{3} = \frac{d}{3} + \frac{1}{3} \geq 1 + \frac{1}{3},
\]
as $p \geq 3$. Hence, $b_j = 0$ if $j \leq 1$. Likewise, if $b_j \neq 0$ then
\[
j \leq \frac{d}{2}+1 + \left(\frac{d}{2} -1\right)\frac{1}{p}-\frac{d}{p^e} < \frac{d}{2}+1 + \left(\frac{d}{2} -1\right)\frac{1}{3} = \frac{2}{3}(d+1),
\]
and so $j \leq d-1$. That is, $b_j \neq 0$ implies $2 \leq j \leq d-1$. In conclusion, if $p\geq 3$, the only sheaves showing up in \autoref{eqn.QuadricDecomposition} are (possibly) in the list
\[
\sO_{\bQ^d}, \sO_{\bQ^d}(1), \ldots ,\sO_{\bQ^d}(d-1), \sS(2), \ldots, \sS(d-1) \quad p\neq 2.
\]
The above list cannot be shortened as it is sharp for $d=3$ and $p \geq 5$ (as well as for $d=4$ and $e \geq 3$). Indeed, we readily see that $\sS(d-1)$ shows up if and only if $d \leq (4p^e-2p^{e-1})/(p^e-p^{e-1}+2)$. In particular, for $d=3$, this always happens unless $(e,p)=(1,3)$. For $d=4$, this is always the case unless either $e=1$ or $(e,p)=(2,3)$. However, for $d \geq 5$, this never happens for $e \gg 0$.

\emph{Let us suppose now that $p=2$.} If $b_j \neq 0$ then
\[
j \geq 1+\frac{1}{2}\big(\lfloor d/2 \rfloor - 1\big) \geq 1+\frac{1}{2}\big(\lfloor 3/2 \rfloor - 1\big) = 1, \quad
j \leq d(1-1/2^e) - \frac{1}{2}\big(\lfloor d/2 \rfloor - 1\big) \leq d-1.
\]
Thus, if $b_j \neq 0$ then $1 \leq j \leq d-1$. Hence, if $p=2$, the only sheaves showing up in \autoref{eqn.QuadricDecomposition} are (possibly) in the list
\[
\sO_{\bQ^d}, \sO_{\bQ^d}(1), \ldots ,\sO_{\bQ^d}(d-1), \sS(1), \sS(2), \ldots, \sS(d-1) \quad p=2,
\]
which cannot be shortened any further as the case $d=3$ shows (for all $e \geq 1$).

Observe that $\sO_{\bQ^d}$ must show up with multiplicity $a_0=1$. Indeed, $\big(F^e_*\omega_{\bQ^d}^{1-q}\big)^{\vee} \cong F^e_* \sO_{\bQ^d}$ and by counting global sections we get $a_0=1$. In particular, $\sE_e$ admits a direct sum decomposition with summands from the list  
\begin{equation} \label{eqn.ListOfSheaves}
 \sO_{\bQ^d}(1), \ldots ,\sO_{\bQ^d}(d-1), \sS(1), \sS(2), \ldots, \sS(d-1),
\end{equation}
where $\sS(1)$ occurs if and only if $p=2$. 
\begin{claim}
$\sS(1)$ is globally generated but not ample, and so $\sS(j)$ is ample for all $j\geq 2$.
\end{claim}
\begin{proof}
To see why $\sS(1)$ is globally generated, use the short exact sequence 
\[
0 \to \sS \to \sO_{\bQ^d}^{\oplus 2^{\lfloor d/2 \rfloor +1}} \to \sS(1) \to 0;
\]
see \cite[\S 1.2]{LangerDAffinityFrobeniusQuadrics} or \cite[(1.3)]{AchingerFrobeniusPushForwardQuadrics}. It remains to explain why $\sS(1)$ is not ample. This can be done by induction on $d$ using how $\sS(1)$ restricts on hyperplane sections; see \cite[\S 2.2.2]{AddingtonSpinorsSheavesCIQuadrics}, and that $\sS(1)$ is not ample for $d=2$. Indeed, using the inductive construction of $\sS$ in terms of matrix factorizations enable us to see that: $\sS(1)= \sS_+(1)\oplus \sS_-(1)$ on $\bQ^{2k}$ restricts to $\sS(1) \oplus \sS(1)$ on $\bQ^{2k-1}=\bQ^{2k} \cap (H\: x_0=x_1)$ and that $\sS(1)$ on $\bQ^{2k+1}$ restricts to $\sS_+(1)\oplus \sS_-(1) = \sS(1)$ on $\bQ^{2k}=\bQ^{2k+1} \cap (H \: x_0=0)$.
\end{proof}

Additionally, $\sS^{\vee} \cong \sS(1)$ \cite[\S 1.1]{LangerDAffinityFrobeniusQuadrics}. In particular, for each sheaf $\sF$ in the above list \autoref{eqn.ListOfSheaves}, $(\sF \otimes \omega_{\bQ^d})^{\vee}$ is ample. Indeed,
\[
\big(\sS(j) \otimes \omega_{\bQ^d} \big)^{\vee } \cong \sS^{\vee} \otimes \sO_{\bQ^d}(d-j) \cong \sS(d-j+1),
\]
which is ample if and only if $j< d$. In other words, $(\sB_e^d)^{\vee}$ is ample. Summing up:
\begin{corollary} \label{cor.Quadrics}
On $\bQ^d$ with $d \geq 3$, $\sB_e^{1,\vee}= \sE_e$ is ample if and only if $p\neq 2$. Further, $\sB_e^{d,\vee} =(\sE_e \otimes \omega)^{\vee}$ is ample for all $p$.
\end{corollary}

\begin{remark}
In principle, one may combine the ideas of \autoref{sec.CONESEXAMPLES} with the computations in \cite{LangerDAffinityFrobeniusQuadrics,AchingerFrobeniusPushForwardQuadrics} to compute $F^e_* R_d$ where $R_d$ is the affine cone over $\bQ^d$, \cf \cite{GesselMonskyTheLimit,TrivediHKFunctionsQuadrics}. See \autoref{rem.ApplicationLocalAlgebraConesRationalNormalCurve}, \autoref{rem.ApplicationLocalAlgebraVeronese}, \autoref{rem.ApplicationLocalAlgebraSegreProduc}. However, this will be pursued elsewhere.
\end{remark}

\section{On the Positivity of Frobenius Trace Kernels} \label{sec.MainSection}

In this section, we study the consequences that positivity conditions on $\sE_{e,X}$ have on the geometry of $X$. Throughout this section, we work on the following setup.
\begin{setup} \label{setup.Positivity}
Let $X$ be a smooth projective variety of dimension $d$. Set $0\neq e \in \bN$ and $\sW_e = \sW_{e,X} \coloneqq  F^e_* \omega_X^{1-q}$, so that $\sE_e = \sE_{e,X} = \ker(\tau^e \:\sW_e \to \sO_X)$.
\end{setup}

\begin{remark}[On the positivity of $\sE_{e}$ with respect to $e$] \label{rem.SequenceOfQuotientsOfE_e}
 We explain why there is a sequence of quotient maps
 \[
  \cdots \twoheadrightarrow \sE_{3,X} \twoheadrightarrow \sE_{2,X}\twoheadrightarrow \sE_{1,X}
 \]
In particular, letting $\sP$ be a positivity property that is inherited to quotients (\eg ampleness, nefness, global generation), if $\sE_{e,X}$ has $\sP$ for all $e \gg 0$ then it has it for all $e>0$. Consider the definitional short exact sequence:
\[
0 \to \sO_X \xrightarrow{F^{e,\#}}F_*^e \sO_X \to \sB_{e,X}^1 \to 0
\]
and push it forward along $F^d$ to obtain
\[
0 \to F^d_*\sO_X \xrightarrow{F^d_*F^{e,\#}}F_*^{d+e} \sO_X \to F^d_*\sB_{e,X}^1 \to 0
\]
which is exact as $F^d$ is affine. Since we also have the short exact sequence 
\[
0 \to \sO_X \xrightarrow{F^{d,\#}}F_*^d \sO_X \to \sB_{d,X}^1 \to 0
\]
we obtain the following one:
\[
0 \to \sB_{d,X}^1 \to \sB_{d+e,X}^1 \to F^d_*\sB_{e,X}^1 \to 0
\]
Dualizing it yields:
\begin{equation}\label{eqn.SESVariationOFEwrte}
    0 \to F^d_* \Bigl(\sE_{e,X} 
\otimes \omega_X^{1-p^d}\Bigr) \to \sE_{d+e,X} \to \sE_{d,X} \to 0
\end{equation}
However, it is unclear to the authors whether $\sE_{e,X}$ being positive for some $e \in \bN$; say $e=1$, implies it for all $e>0$. The reason is that it is unclear how to preserve positivity along Frobenius pushforwards. Also, see \autoref{rem.NaiveAmplenessForOneiMPLIESfORALL} and \autoref{que.AmplenessForSomeImplesAll} below.
\end{remark}

\subsection{Global generation}
In this subsection, we rely on \cite{MurayamaFrobeniusSeshadri,MurayamaPhDThesis}. 

\begin{lemma} \label{lem.GlobalGenerationOfsF_e}
Working in \autoref{setup.Positivity}, $\sW_e$ is globally generated if and only if $\sE_e$ is globally generated and $X$ is $F$-split.
 \end{lemma}
\begin{proof}
If $X$ is $F$-split, then $\sW_e \cong \sE_e \oplus \sO_X$ and so it is globally generated if (and only if) so is $\sE_e$. Conversely, suppose that $\sW_e$ is globally generated, then there are surjections $\sO_X^{\oplus n} \twoheadrightarrow \sW_e \twoheadrightarrow \sO_X$. We then have $n$ morphisms $\sO_X \to \sO_X$, which amounts to having $n$ global sections of $\sO_X$, i.e., $n$ elements of $\kay$. By surjectivity, at least one of these scalars must be nonzero. Thus, $H^0(X,\sW_e) \to H^0(X,\sO_X)$ is surjective. An element in $H^0(X,\sW_e)$ that is mapped to $1$ corresponds to a splitting of \autoref{eqn.SESDefiningE_e}. Therefore, $X$ is $F$-split and $\sE_e$ is globally generated.
\end{proof}

 \begin{definition}
Let $\sF$ be a locally free sheaf on a scheme $X$. One says that $\sF$ separates $l$-jets at a closed point $x\in X$ if the canonical restriction-of-sections map
\[
H^0(X,\sF) \to H^0\bigl(X,\sF \otimes \sO_X/\mathfrak{m}_x^{l+1}\bigr)
\]
is surjective, where $\m_x$ denotes the ideal sheaf defining $x$. Further, $\sF$ is said to separate $l$-jets if it separates $l$-jets at every closed point. Likewise, $\sF$ separates $q$-Frobenius $l$-jets at $x \in X$ if
\[
H^0\bigl(X,\sF\bigr) \to H^0\Bigl(X,\sF \otimes \sO_X \big/\bigl(\fram_x^{l+1}\bigr)^{[q]}\Bigr)
\]
is surjective. If this holds for all $x\in X$, one says that $\sF$ separates $q$-Frobenius $l$-jets.
\end{definition}

\begin{remark}
A locally free sheaf is globally generated if and only if it separates $0$-jets.
\end{remark}

\begin{lemma} \label{lem.l-jetsSeparationTranslation}
Let $X$ be an $F$-finite scheme and $x\in X$ be a closed point. An invertible sheaf $\sL$ on $X$ separates $q$-Frobenius $l$-jets at $x$ if and only if $F^e_* \sL$ separates $l$-jets at $x\in X$.
\end{lemma}
\begin{proof}
By definition, $\sL$ separates $q$-Frobenius $l$-jets at $x$ if and only if the restriction map
\[
H^0\bigl(X,\sL\bigr) \to H^0\Bigl(X,\sL \otimes \sO_X \big/\bigl(\fram_x^{l+1}\bigr)^{[q]}\Bigr)
\]
is surjective. Nevertheless, the surjectivity of this map is equivalent to the surjectivity of
\[
H^0\bigl(X,F^e_*\sL \bigr) \to H^0\Bigl(X,F^e_* \Bigl( \sL \otimes \sO_X \big/\bigl(\fram_x^{l+1}\bigr)^{[q]} \Bigr)\Bigr).
\]
However,
\[
F^e_* \Bigl( \sL \otimes \sO_X \big/\bigl(\fram_x^{l+1}\bigr)^{[q]} \Bigr) = \bigl(F^e_* \sL\bigr) \otimes \sO_X / \mathfrak{m}_x^{l+1}.
\]
Therefore, $\sL$ separates $q$-Frobenius $l$-jets at $x$ if and only if the restriction map
\[
H^0\bigl(X, F^e_*\sL \bigr) \to H^0\Bigl(X,F^e_*\sL \otimes \sO_X \big/\fram_x^{l+1}\Bigr)
\]
is surjective, which means that $F^e_*\sL$ separates $l$-jets at $x$.
 \end{proof}
 
\begin{proposition} \label{prop.FanoByGlobalGeneration}
Working in \autoref{setup.Positivity}, if $\sE_{e}$ is globally generated and $X$ is $F$-split then $X$ is Fano.
\end{proposition}
\begin{proof}
Note that $\omega_X^{1-q}$ separates $q$-Frobenius $0$-jets. Indeed, by \autoref{lem.l-jetsSeparationTranslation}, this means that $\sW_e$ separates $0$-jets, i.e., it is globally generated. However, this follows from \autoref{lem.GlobalGenerationOfsF_e}.  On the other hand, by \cite[Proposition 2.5 (ii)]{MurayamaFrobeniusSeshadri}:
\[
\varepsilon^l_F\bigl(\omega_X^{-1};x\bigr) \geq \sup_{m,e} \frac{q-1}{m/(l+1)}
\]
where the supremum traverses all $m,e$ such that $\omega_X^{-m}$ separates $q$-Frobenius $l$-jets at $x$. Notice that we are using the trivial inequality in \cite[Proposition 2.5 (ii)]{MurayamaFrobeniusSeshadri}; which does not require $X$ to be Fano. Therefore, 
\[
\varepsilon^0_F\bigl(\omega_X^{-1};x\bigr) \geq (q-1)/(q-1) = 1,
\]
for all points $x\in X$. Nonetheless, $\varepsilon\bigl(\omega_X^{-1};x\bigr) \geq \varepsilon^l_F\bigl(\omega_X^{-1};x\bigr)$ for all $l$ and all $x\in X$ (notice that $X$ is regular, and this inequality does not require $X$ to be Fano); see \cite[Proposition 2.9]{MurayamaFrobeniusSeshadri}. Hence, $\varepsilon\bigl(\omega_X^{-1};x\bigr) \geq 1$ for all points $x$. According to \cite[Corollary 7.2.7]{MurayamaPhDThesis}, this suffices to prove that $\omega_X^{-1}$ is ample, and so that $X$ is Fano.
\end{proof}

\begin{remark}
With notation as in \autoref{prop.FanoByGlobalGeneration}, let $\sL$ be an invertible sheaf on $X$. Since $(F^e_* \sL)^{\vee} \cong F_*^e(\sL^{-1} \otimes \omega_X^{1-q})$, we have that $(F^e_* \sL)^{\vee}$ is globally generated if and only if $\sL^{-1} \otimes \omega_X^{1-q}$ separates $q$-Frobenius $0$-jets. Therefore, the same argument as in \autoref{prop.FanoByGlobalGeneration} proves that if $(F^e_* \sL^{q-1})^{\vee}$ is globally generated then $\sL^{-1} \otimes \omega_X^{-1}$ is ample.
\end{remark}

\subsection{Ampleness and numerical effectiveness} 
We have the following result.

\begin{proposition} \label{thm.AmpleImpliesFano}
Working in \autoref{setup.Positivity}, if $\sE_e$ is nef then so is $\omega_X^{-1}$. Further, if $\sE_e$ is ample then $X$ is Fano.
\end{proposition}
\begin{proof}
Pulling back \autoref{eqn.SESDefiningE_e} along $F^e$ yields a short exact sequence
\[
0 \to F^{e,*}\sE_e \to F^{e,*}\sW_e \to \sO_X \to 0.
\]
Since $F^e$ is finite, $F^{e,*}\sE_e$ is nef (resp. ample) if so is $\sE_e$. Thus, if $\sE_e$ is nef, $F^{e,*}\sW_e$ is an extension of nef locally free sheaves and so it is nef as well \cite[Lemma 6.2.12 (i)]{LazarsfeldPositivity2}. Thus, the canonical morphism $F^{e,*}F^e_* \omega_X^{1-q} \to \omega_X^{1-q}$ realizes $\omega_X^{1-q}$ as a quotient of a nef locally free sheaf and hence $\omega_X^{1-q}$ is nef \cite[6.1.2 (i)]{LazarsfeldPositivity2}. Hence, $\omega_X^{-1}$ is nef (for one of its powers is nef).

The above argument fails in showing that the ampleness of $\sE_e$ is inherited by $\omega_X^{-1}$ because $\sO_X$ is not ample. To bypass this, we prove that the composition
\begin{equation} \label{eqn.KeySurjectiveComposition}
F^{e,*}\sE_e \to F^{e,*}\sW_e \to \omega_X^{1-q}
\end{equation}
is surjective. Consequently, if $\sE_e$ is a ample, a power of $\omega_X^{-1}$ would be realized as the quotient of an ample locally free sheaf and so $\omega_X^{-1}$ would be ample. 

In order to prove that \autoref{eqn.KeySurjectiveComposition} is surjective, we may restrict to stalks. Let $x\in X$ be a point. Twisting \autoref{eqn.SESDefiningE_e} by $\sO_{X,x}$ yields the following short exact sequence of $\sO_{X,x}$-modules
\[
0 \to \sE_{e,x} \to F^e_* \sO_{X,x} \xrightarrow{\kappa^e_x} \sO_{X,x} \to 0
\]
where $\kappa^e_x \: F^e_* \sO_{X,x} \to \sO_{X,x}$ is the Cartier operator associated to the local regular (and so Gorenstein) ring $\sO_{X,x}$; see \autoref{rem.LocalDescritptionTraces}. For notation ease, let us write $\sO_{X,x}^{1/q}$ instead of $F^e_* \sO_{X,x}$. Thus, pulling back along Frobenius gives the following short exact sequence
\[
0 \to \sO_{X,x}^{1/q} \otimes \sE_{e,x} \to \sO_{X,x}^{1/q} \otimes \sO_{X,x}^{1/q} \xrightarrow{\sO_{X,x}^{1/q} \otimes \kappa_x^e} \sO_{X,x}^{1/q} \to 0.
\]
On the other hand, the localization of $F^{e,*}\sW_e \to \omega_X^{1-q}$ at $x$ corresponds to the diagonal homomorphism $\delta \: \sO_{X,x}^{1/q} \otimes \sO_{X,x}^{1/q} \to \sO_{X,x}^{1/q}$ realizing $\sO_{X,x}^{1/q}$ as an $\sO_{X,x}$-algebra. Therefore, it suffices to prove that the composition
\[
\sO_{X,x}^{1/q} \otimes \sE_{e,x} \to \sO_{X,x}^{1/q} \otimes \sO_{X,x}^{1/q} \xrightarrow{\delta} \sO_{X,x}^{1/q}
\]
is surjective. By $\sO_{X,x}^{1/q}$-linearity, it suffices to show that $1=1^{1/q} \in \sO_{X,x}^{1/q}$ belongs to the image. Note that $1^{1/q} \in \sO_{X,x}^{1/q}$ belongs to $\sE_{e,x}$ as $\kappa^e_x(1^{1/q})=0$; see \autoref{rem.LocalDescritptionTraces}. Then, the image of $1^{1/q} \otimes 1^{1/q} \in \sO_{X,x}^{1/q} \otimes \sE_{e,x}$ is $\delta(1^{1/q} \otimes 1^{1/q}) = 1^{1/q} \in \sO_{X,x}^{1/q}$; as desired. 
\end{proof}

\begin{scholium}
Work in the setup of \autoref{thm.AmpleImpliesFano}. Let $\sP$ be a (positivity) property on locally free sheaves that can be induced via quotients and symmetric powers and is preserved under finite pullbacks. If $\sE_e$ satisfies $\sP$ then so does $\omega_X^{-1}$.
\end{scholium}
\begin{proof}
In the proof of \autoref{thm.AmpleImpliesFano}, we showed that there is a surjective morphism $F^{e,*}\sE_{e} \to \omega_X^{1-q}$. Hence, if $\sE_e$ satisfies $\sP$ then so does $F^{e,*}\sE_{e}$ by preservation under finite pullback. Then $\omega_X^{1-q}$ satisfies $\sP$ by induction via quotients and so does $\omega_X^{-1}$ via induction by powers.
\end{proof}

\begin{corollary} \label{cor.MainCorDim1}
Work in \autoref{setup.Positivity} with $d=1$. Then $\sE_e$ is ample if and only if $X \cong \bP^1$.
\end{corollary}

\begin{remark} \label{rem.NaiveAmplenessForOneiMPLIESfORALL}
In \autoref{rem.SequenceOfQuotientsOfE_e}, we had mentioned that it is unclear that $\sE_e$ being ample (or, say, nef) for some $e$ implies that it is for all $e \in \bN$. One may wonder whether the quotient map \autoref{eqn.KeySurjectiveComposition} may help to elucidate this. Combining it with \autoref{eqn.SESVariationOFEwrte} and using the projection formula yields the exact sequence
\[
\sE_d \otimes F^d_* \sE_e \to \sE_{d+e} \to \sE_{d} \to 0
\]
However, due to the pushforward $F^d_*$, it is unclear whether $\sE_{e+d}$ is ample if so are $\sE_e$ and $\sE_d$. In fact, it is not true in general that $F^d_*\sE_e$ nor $\sE_d \otimes F^d_* \sE_e$ are ample if so are $\sE_e$ and $\sE_d$. For instance, for $X = \bP^1$, we have that $\sE_e = \sO(1)^{\oplus (q-1)}$ but $F^d_* \sO(1) = \sO^{\oplus 2} \oplus \sO(-1)^{\oplus(p^d-2)}$. However, one can still ask:
\end{remark}

\begin{question} \label{que.AmplenessForSomeImplesAll}
    Suppose that $\sE_1$ is ample (resp. nef). Is it true that $F_*(\sE_1 \otimes \omega^{1-p})$ is ample (resp. nef)?
\end{question}

\subsection{Extremal contractions} \label{Main.SubsectionFibrationsContractions}
In studying when $\sE_e$ is ample, \autoref{thm.AmpleImpliesFano} let us restrict ourselves to Fano varieties. To narrow this down further, we investigate the conditions that the ampleness of $\sE_{e,X}$ imposes on extremal contractions of $X$. We start off with a general remark for smooth fibrations. By a \emph{fibration}, we mean a proper morphism $f\: X \to S$ with connected fibers (i.e., $f^{\#} \:\sO_X \to f_* \sO_X$ is an isomorphism).

\begin{proposition} \label{lem.Fibrations}
Let $f\: X \to S$ be a fibration between smooth varieties whose general fiber is smooth and fix $0 \neq e \in \bN$. If $\sE_{e,X}$ is ample and $\dim S >0$ then the general fiber of $f$ is zero-dimensional. In particular, all fibers are zero-dimensional if $f$ is further flat. 
\end{proposition}
\begin{proof}
There is an open $\emptyset \neq U \subset S$ such that the restriction $f_U \: X_U \to U$ is a smooth fibration between smooth varieties (using generic flatness, openess of the regular locus of $S$, and the given hypothesis of smoothness of the general fiber). By \autoref{rem.NatCarOp}, there is a surjective morphism $\varepsilon_{e,X_U/U}:\sE_{e,X_U} \to f_U^* \sE_{e,U}$. Its pullback along a fiber $g\:X_s \to X$ at a closed point $s \in U(\kay)$ (so $X_s \subset X_U$) will be a surjection of the form 
\[
g^* \sE_{e,X } \to \sO_{X_s}^{\oplus (q^{\dim S} -1)}.
\] Therefore, if $\sE_{e,X}$ is ample then so is $\sO_{X_s}^{\oplus (q^{\dim S} -1)}$. Hence, $\dim X_s=0$ as $\dim S >0$.
\end{proof}

We had seen above (see \autoref{cor.(-1)CurvesRuleOut}) that if $\sE_{e,X}$ is ample for a surface $X$, then $X$ contains no $(-1)$-curve. We then obtain the following.

\begin{corollary} \label{cor.MainCorDim2}
Work in \autoref{setup.Positivity} with $d=2$. Then $\sE_e$ is ample if and only if $X \cong \bP^2$.
\end{corollary}
\begin{proof}
Suppose that $\sE_{e,X}$ is ample. By \autoref{cor.(-1)CurvesRuleOut}, $X$ contains no $(-1)$-curve. Therefore, any extremal contraction $X \to S$ is a Mori fibration. More precisely, $f\:X \to C$ is either a ruled surface or $X \cong \bP^2$. We rule out the ruled surface case by using \autoref{lem.Fibrations}.\footnote{Note that we have done this explicitly in \autoref{ex.HirzebruchSurfaces}.}
\end{proof}

With the above proof of \autoref{cor.MainCorDim2} in place, we see how to proceed for threefolds. Fortunately, we have a good description of extremal contractions on smooth threefolds. We recall the following fundamental result, which was originally due to S.~Mori in characteristic zero in his seminal work \cite{MoriThreefoldsWhoseCanonicalBundleNotNef} and later generalized to all characteristics by J.~Koll\'ar; see \cite[Main Theorem]{KollarExtremalRaysOnSmoothThreefolds}.

\begin{theorem}[Koll\'ar--Mori's description of smooth threefold extremal contractions] \label{thm.StructureExtremalContractionsDim3}
Let $X$ be a smooth threefold and $f\: X \to S$ be an extremal contraction. If $f$ is birational then it is one of the following divisorial contractions with exceptional divisor $E \subset X$:
\begin{enumerate}
    \item $S$ is smooth and $f$ is the blowup along a smooth curve 
    $C \subset S$. In this case, $f_C \:E \to C$ is a smooth minimal ruled surface.
    \item $S$ is smooth and $f$ is the blowup at a point $s\in S$. In this case, $E \cong \bP^2$ with normal bundle corresponding to $\sO_{\bP^2}(-1)$.
    \item $S$ has exactly one singular point $s \in S$ and $f$ is the blowup of $S$ at $s$. Moreover, one of the following three cases holds:
    \begin{enumerate}
    \item  $\hat{\sO}_{S,s} \cong \kay \llbracket x,y,z \rrbracket^{\bZ/2} \cong \kay \llbracket x^2,y^2,z^2,xy,yz,zx \rrbracket \eqqcolon R_1$, where $\bZ/2$ acts via the involution $(x,y,z) \mapsto (-x,-y,-z)$, and $E \cong \bP^2$ with normal bundle $\sO_{\bP^2}(-2)$.
    \item  $\hat{\sO}_{S,s} \cong \kay \llbracket x,y,z,t \rrbracket/(xy-z^2-t^3) \eqqcolon R_2$ and $E$ is isomorphic to the singular quadric cone $Q \subset \bP^3$  with normal bundle corresponding to $\sO_Q(-1)$.
     \item $\hat{\sO}_{S,s} \cong \kay \llbracket x,y,z,t \rrbracket/(xy-zt) \eqqcolon R_3$ and $E\cong \bQ^2 \subset \bP^3$ with normal bundle corresponding to $\sO_{\bQ^2}(-1)$.
    \end{enumerate}
\end{enumerate}
If $f$ is not birational then it corresponds to one of the following Fano fibrations:
\begin{enumerate}[label=(\arabic*)]
    \item $S$ is a smooth surface and $f \: X \to S$ is a flat conic bundle (i.e. every fiber is isomorphic to a conic in $\mathbb{P}^2$). If $p \neq 2$, the general fiber of $f$ is smooth.
    \item $S$ is a smooth curve and every fiber of $f \: X \to S$ is irreducible and every reduced fiber is a (possibly nonnormal) del Pezzo surface. However, the general fiber is a normal del Pezzo surface \cite{FanelliSchroerDelPezzoSurfacesAndMoriFiberSpacesInPosChar}.
    Further, if $p> 7$, the general fiber of $f$ is a smooth del Pezzo surface \cite{PatakfalviWaldronSingularitiesOfGeneralFibersLMMP}. Noteworthy, $f$ is necessarily flat; see \cite[III, Proposition 9.7]{Hartshorne}.\footnote{If $p=2$, the generic fiber need not be smooth \cite{FanelliSchroerDelPezzoSurfacesAndMoriFiberSpacesInPosChar}. We do not know of examples if $p=3,5,7$.}
    \item $S$ is a point and so $X$ is a Fano variety of Picard rank $1$.
\end{enumerate}
\end{theorem}

In this way, the ampleness of $\sE_{e,X}$ rules out most possible extremal contractions that $X$ can undergo:

\begin{proposition} \label{prop.RulingOutCases}
With notation as in \autoref{thm.StructureExtremalContractionsDim3}, suppose that $\sE_{e,X}$ is ample (for some $0\neq e \in \bN$) but the Picard rank $\rho(X) \geq 2$. Then $f$ is either as in case ii. or a wild del Pezzo fibration (so $p \leq 7$); i.e., as in case (2) where the geometric generic fiber (although normal) is not smooth.
\end{proposition}
\begin{proof}
The tame (i.e. non-wild) instances of (1) and (2) are ruled out by \autoref{lem.Fibrations}. Next, we explain why there cannot be wild conic fibrations (which only happen if $p=2$). Suppose $p=2$ and that $X$ admits a wild conic fibration $f\: X \to S$. Fortunately, these have been classified in \cite[Corollary 8]{MoriSaitFanoThreefoldsWithWildConicBundleStructures}. There are two cases, which we show next to be impossible, yielding the sought contradiction. The cases are as follows.

\emph{First case:} $X \subset \bP^2 \times \bP^2$ is a divisor of bidegree $(1,2)$ and $f\: X \to S$ is the projection into the second factor $\bP^2$ (e.g. \cite[Example 4.12]{KollarExtremalRaysOnSmoothThreefolds}). However, the projection $g\: X \to \bP^2$ onto the first factor is a smooth $\bP^1$-fibration (see \cite[final case in \S2.3]{SaitoFanothreefoldsWithPicardNumber2PosChar}) whose existence violates \autoref{lem.Fibrations}.

\emph{Second case:} $f$ is given by $X \subset \bP(\sO(1,0) \oplus \sO(0,1) \oplus \sO) \to \bP^1 \times \bP^1$ where $X$ is a smooth prime divisor in the linear system  $|\sO_{\bP}(2)|$. However, by \cite[Remark 10]{MoriSaitFanoThreefoldsWithWildConicBundleStructures}, $X$ is also the blowup of the smooth quadric threefold $\bQ^3 \subset \bP^4$ along the union of two disjoint smooth conics $C_1, C_2 \subset \bQ^3$ (e.g. \cite[Example 5.3]{SaitoFanothreefoldsWithPicardNumber2PosChar}). Nonetheless, we know that these cannot exist either if $\sE_{e,X}$ is to be ample by \autoref{lem.FrobPushForwardLocalgeneralBlowups}. 

We see that cases (a) and (b) are impossible by applying \autoref{lem.FrobPushForwardLocalgeneralBlowups}---just as we did in the proof of \autoref{cor.(-1)CurvesRuleOut} (\cf proof of \autoref{cor.MainCorDim2}). Thus, we are left with ruling out cases i. and iii. Inspired by the previous two cases, our strategy will be to pullback $\sE_{e,X}$ to the exceptional divisor of the blowup arguing that such pullback is not ample. We do it by computing the restriction explicitly and showing it has a free direct summand. Since the argument is local around the singular point $s$, we may replace $S$ by any projective threefold realizing that singular point. Then, we compute $F^e_*\sO$ (and so $\sE_e$) for the blowup of that threefold at the singular point and subsequently its pullback to the exceptional divisor. We start off with case i. We may consider $S$ to be projective cone over the Veronese surface $\bP^2 \cong V\subset \bP^5$; see \autoref{sec.CONESEXAMPLES}. Then, if $X \to S$ is the blowup of $s$ at its vertex $s$, then $\sE_{e,X}$ is not ample and $\hat{\sO}_{S,s} \cong R_1$; see \autoref{sec.VeroneseEmbeddings}. Similarly, for case iii., we may consider $S$ to be projective cone over the Segre embedding $\bP^1 \times \bP^1 \cong \bQ^2 \subset \bP^3$. If $s \in S$ denotes the vertex singularity, then $\hat{\sO}_{S,s} \cong R_3$ and its blowup $X \to S$ is so that $\sE_{e,X}$ is not ample as demonstrated in \autoref{sec.SegreEmbeddings}.
\end{proof}

Unfortunately, the authors do not know how to rule out the remaining cases of \autoref{prop.RulingOutCases}. For example, case ii. is quite different from the other two cases of (c). To bypass this issue, we are going to take a closer look at the structure of extremal contractions of smooth Fano threefolds as pioneered by \cite{MoriMukaiClassificationFanoThreefoldsWithB2geq2I,MoriMukaiClassificationFanoThreefoldsWithB2geq2II,MoriMukaiClassificationFanoThreefoldsWithB2geq2III}, which were done in characteristic zero. For the positive characteristic case, see \cite{SaitoFanothreefoldsWithPicardNumber2PosChar,MoriSaitFanoThreefoldsWithWildConicBundleStructures}, \cf \cite{ShepherdBArronFanoThreefoldInPosChar,MegyesiFanoThreefoldInPosChar,TakeuchiSomeBirationalMapsOfFanoThreefolds}. Now, we need not the full strength of those analyses, as all we need is a result of the form \cite[Corollary 1.3]{WisniewskiOnContractionsOfExtremalRaysOfFanomanifolds} or say (much weaker versions of) \cite[Theorem 5]{MoriMukaiClassificationFanoThreefoldsWithB2geq2I}, \cite[Theorem 1.6]{MoriMukaiClassificationFanoThreefoldsWithB2geq2II}, or \cite{MoriMukaiClassificationFanoThreefoldsWithB2geq2III}. In this regard, we have the following. The ideas are those of Mori--Mukai in op. cit (so no originality is claimed). However, we provide a proof for the lack of an adequate reference in positive characteristics.

\begin{proposition}\label{prop.NarrowingDownCasesFanoCase}
Let $X$ be a smooth Fano threefold of Picard rank $\rho(X) \geq 2$. Then, $X$ admits an extremal contraction $f\: X \to S$ that is either as in case (a) or as in case (1) of \autoref{thm.StructureExtremalContractionsDim3}.
\end{proposition}
\begin{proof}
Let $\Gamma$ denote the (closed) cone of curves of $X$, which is a finite polyhedral cone as $X$ is a smooth Fano threefold. Let $R_1, \ldots, R_n$ be the extremal rays of $\Gamma$ with corresponding extremal contractions $f_i \: X \to S_i$; see \cite[\S3.7]{KollarMori}. Suppose, for the sake of contradiction, that none of the $f_i$ is a smooth blowup (case (a)) nor a conic bundle (case (1)). 

Let $\Delta \subset \Gamma$ be the subcone spanned by those extremal rays that produce divisorial contractions (only of the types (b) and (c) by assumption). For notation ease, let us say that these are the first $m$ extremal rays (if any). Let $E_1, \ldots, E_m \subset X$ denote the corresponding exceptional divisors (if any). The first observation is that these divisors are pairwise disjoint. Indeed, let $L \coloneqq E_i \cap E_j$ for $i \neq j$. Then, on the one hand, $L \cdot E_i =  L \cdot E_i|_{E_i} <0$ as $\sO_{E_i}(E_i)$ is always negative (according to \autoref{thm.StructureExtremalContractionsDim3}). On the other hand, $L \cdot E_i \geq 0$ as curves in $E_j$ move (see the options in \autoref{thm.StructureExtremalContractionsDim3}). Then, one readily sees that $Z \cdot E_i \leq 0$ for all $Z \in \Delta$. In particular, $(-K_X)^2 \notin \Delta$ and so $\Delta \neq \Gamma$ (i.e., $n >m$). 

In conclusion, $f\coloneqq f_n \: X \to S_n \eqqcolon S$ must be a del Pezzo fibration (with normal general fiber). Since $H^1(X,\sO_X)=0$ (\cite[Corollary 3.7]{KawakamiKVVForFano3FoldsPosChar}, \cite[Corollary 1.5]{ShepherdBArronFanoThreefoldInPosChar}) and $f_* \sO_X = \sO_S$, then $H^1(S,\sO_S) = 0$ (as the Leray spectral sequence yields $H^1(S,\sO_S) \subset H^1(X,\sO_X)$). In particular, $S = \bP^1$ and so $\rho(X) = \rho(S)+1=2$. Let $g \: X \to S$ be the other extremal contraction. By assumption, it is either another del Pezzo fibration or a blowup at a point. If it were another del Pezzo fibration, then it would give a surjective map $f \times g \: X \to \bP^1 \times \bP^1$ violating that $\rho(X)=2$. Hence $g$ must be a blowup at a point (i.e. of type (b) or (c)). Let $E$ be its exceptional divisor, which is isomorphic to either  $\bP^2$, $Q$ (singular quadric cone), or the smooth quadric surface $\bQ^2 \cong \bP^1 \times \bP^1$ (according to \autoref{thm.StructureExtremalContractionsDim3}). In the first two cases, it is clear that $f$ would have to contract $E$ to a point, and so $E$ would be a fiber of $f$, contradicting that $\sO_E(E)$ is negative. The same holds in the third case as well, yet a little argument is needed for why $f$ contracts $E \cong \bQ^2 \cong \bP^1 \times \bP^1$ to a point. The key observation is that the ruling lines $x \times \bP^1$ and $\bP^1\times y$ (for closed points $x,y \in \bP^1$) are numerically equivalent inside $X$. Hence, if either of them does not intersect the fibers of $f$ then neither does the other. In particular, the restriction of $f$ to $E$ cannot be one of the canonical projections $\bP^1 \times \bP^1 \to\bP^1$ and hence it got to be a contraction to a point.
\end{proof}

\begin{theorem} \label{thm.MainTheorem}
Work in \autoref{setup.Positivity} with $d=3$. Then, if $\sE_{e}$ is ample then $X$ is a Fano threefold of Picard rank $1$.
\end{theorem}
\begin{proof}
By \autoref{thm.AmpleImpliesFano}, $X$ is a Fano threefold. Putting \autoref{prop.RulingOutCases} and \autoref{prop.NarrowingDownCasesFanoCase} together yields $\rho(X) = 1$.
\end{proof}

\begin{remark}[Converse of \autoref{thm.MainTheorem}]
As we saw in \autoref{cor.Quadrics}, the converse of \autoref{thm.MainTheorem} seems to be rather subtle. In principle, since we may have a classification of Fano threefolds of Picard rank $1$ \cite{ShepherdBArronFanoThreefoldInPosChar,MegyesiFanoThreefoldInPosChar,TakeuchiSomeBirationalMapsOfFanoThreefolds,KawakamiKVVForFano3FoldsPosChar}, one may analyze the ampleness of $\sE_{e,X}$ case by case. Of course, the remaining cases are those of index $2$ (also known as del Pezzo threefolds) and those of index $1$ where the former is arguably the most tractable by direct analysis. Recall that the index-$2$ case includes $\bP^6 \cap \bG(2,5)$ (inside $\bP^9$ with respect to the Pl\"{u}cker embedding $\bG(2,5) \subset \bP^9$), the complete intersection of two smooth quadrics in $\bP^5$, and the smooth cubic hypersurface in $\bP^4$. These seem to be the easiest cases that might be computed explicitly. For instance, the computations in \cite{RaedscheldersSpenkoVandenBerghFrobeniusMorphismInInvariantTheory} may be very useful to answer this for $\bP^6 \cap \bG(2,5)$. In general, a different approach seems necessary. We do not attempt to pursue this here.
\end{remark}

\begin{remark}[Higher dimensions]
If the main results in \cite{WisniewskiOnContractionsOfExtremalRaysOfFanomanifolds} were to hold in positive characteristics, we may reduce the study of the ampleness of $\sE_{e,X}$ and extremal contractions in dimensions $\geq 4$ to those where divisors are not contracted (e.g., flipping contractions which we have not discussed so far) and of wild conic bundles. For instance, in Mori--Mukai's terminology, we may assume our Fano variety to be primitive by \autoref{lem.FrobPushForwardLocalgeneralBlowups} (which most likely sets an upper bound on the Picard rank in general). We will not attempt this here.
\end{remark}

\subsection{Further remarks}
To conclude, we would like to add some final comments regarding the positivity of the Frobenius trace kernels. For example, why is the positivity of $\sE_{e,X}$ so (seemingly) difficult to study for a hypersurface $X\subset \bP^d$? Is there some adjunction principle that may help?

\subsubsection{Hypersurfaces, complete intersections, and smooth blowups}
Let $X$ be a smooth variety and $Y \subset X$ be a smooth irreducible closed subvariety defined by $\sI \subset \sO_X$ (so $\sI$ is locally generated by $\codim(Y,X)$ elements; see \cite[II, Theorem 8.17]{Hartshorne}). By adjunction, $\omega_Y \cong \omega_X \otimes \det \sN_{Y/X}$, where $\sN_{Y/X} \coloneqq \ssHom_Y\bigl(\sI/\sI^2, \sO_Y\bigr)$ is the normal bundle of $Y$ in $X$. We mention next how $\kappa^e_Y \: F^e_* \omega_Y \to \omega_Y$ is related to $\kappa^e_X \: F^e_* \omega_X \to \omega_X$ through adjunction. There is a commutative diagram of exact sequences
\[
\xymatrix{
0 \ar[r] & \sI \otimes F^e_* \omega_X^{1-q} \ar[r] \ar[d]^-{\sI \otimes \tau_X^e} & F^e_* \big((\sI^{[q]} : \sI) \otimes \omega_X^{1-q} \big) \ar[r] \ar[d]^-{\tau_{Y/X}^e} &  F^e_* \omega_Y^{1-q} \ar[r] \ar[d]^-{\tau_Y^e} & 0\\
0 \ar[r] & \sI \ar[r] & \sO_X \ar[r] &  \sO_Y \ar[r] & 0
}
\]
where $\tau_{Y/X}^e$ is the restriction of $\tau_X^e \: F^e_* \omega_X^{1-q} \to \sO_X$ via the natural inclusion $(\sI^{[q]} : \sI) \otimes \omega_X^{1-q} \subset \omega_X^{1-q}$. Further, if $\sI$ is locally generated by a regular sequence $f_1,\ldots,f_m$, then $\sI^{[q]}:\sI$ is generated by $f_1^q,\ldots,f_m^q, (f_1 \cdots f_m)^{q-1}$; see \cite[Proposition (d) p.110]{HochsterMATH615LectureNotesWinter2010}. The above diagram works via the isomorphism of $\sO_Y$-modules
\[
\frac{\sI^{[q]} : \sI}{\sI^{[q]}}  \xrightarrow{\cong} \bigl( \det \sN_{Y/X} \bigr)^{1-q} = \bigl( \det \sI/\sI^2\bigr)^{q-1},
\]
which is defined by $g \cdot \bigl(f_1 \cdots f_m \bigr)^{q-1} \mapsto g \cdot \bigl(f_1 \wedge \cdots \wedge f_m\bigr)^{q-1}$ on an open neighborhood $U$ where $\sI|_U$ is defined by a regular sequence $f_1,\ldots,f_m$.

By letting $\sE_{e,Y/X}$ denote the kernel of $\tau_{Y/X}^e\: F^e_* \big((\sI^{[q]} : \sI) \otimes \omega_X^{1-q} \big) \to \sO_X$, we obtain an exact sequence
\begin{equation} \label{eqn.cartierAdjunction}
    0 \to \sI \otimes \sE_{e,X} \to \sE_{e,Y/X} \to \sE_{e,Y} \to 0
\end{equation}
If $Y \subset X$ is a divisor, $\sI^{[q]}:\sI = \sO_X((1-q)Y)$ and $\sE_{e,Y/X}$ is the kernel of 
\[(F^e_* \sO_X((q-1)Y))^{\vee} \cong F^e_* \sO_X((1-q)(K_X+Y)) \to \sO_X\] 
where $K_X$ is a canonical divisor on $X$. By the same argument of \autoref{thm.AmpleImpliesFano}, if $\sE_{e,Y/X}$ is ample then $-(K_X+Y)$ is ample (implying that $X$ and $Y$ are both Fano). In this case, \autoref{eqn.cartierAdjunction} takes the form
\[
0 \to \sE_{e,X} \to \sO_X(Y) \otimes \sE_{e,Y/X} \to \sO_Y(Y) \otimes \sE_{e,Y} \to 0
\]
Now, if $X = \bP^d$ and $Y$ is a smooth hypersurface of degree $n \leq q$, then
\[
\sE_{e,Y/X}(n-1) \cong \bigoplus_{i=1}^{d} \sO_X(i)^{\oplus a_{i,q-n;d,e}}.
\]
Therefore,  $\sE_{e,Y}(n-2)$ is globally generated and so $\sE_{e,Y}(n-1)$ is ample. In general, it is seemingly difficult to extract more information about $\sE_{e,Y}$ from this. For instance, whether or not $\sO_Y(n-2) \otimes \sE_{e,Y}$ is ample is subtle and is not true in general in view of \autoref{ex.QuadricsAnalysis}. Also, for projective spaces, $\sE_e(-1)$ is globally generated while this is never true for quadrics.

Let us now mention the case of smooth blowups. With $Y \subset X$ as above, suppose that $\codim(Y,X) = r \geq 2$ and let $\pi \: \Tilde{X} \to X$ be the blowup of $X$ along $Y$ with exceptional divisor $Y' \subset \tilde{X}$. Then, there is an exact sequence
\[
0 \to \sO_{\tilde{X}}(-Y') \otimes F^e_* \sO_{\Tilde{X}}((1-q)K_{\tilde{X}}) \to F^e_* \sO_{\Tilde{X}}((1-q)(K_{\tilde{X}}+Y')) \to F^e_* \omega_{Y'}^{1-q} \to 0.
\]
Equivalently,
\[
0 \to F^e_* \sO_{\Tilde{X}}((1-q)K_{\tilde{X}}) \to F^e_* \sO_{\Tilde{X}}((1-q)K_{\tilde{X}}+Y') \to \sO_{\bP}(-1) \otimes F^e_* \omega_{Y'}^{1-q} \to 0
\]
as $\pi|_{Y'} \: Y' \to Y$ is the projective bundle $\bP(\sI/\sI^2) \to Y$ and $\sN_{Y'/\tilde{X}} = \sO_{\bP}(-1)$. It is unclear to us how this could help in studying the positivity of $\sE_{e,\tilde{X}}$, say by restricting it to $Y'$. This is why we needed to rely on \autoref{lem.FrobPushForwardLocalgeneralBlowups}.

\subsubsection{Asymptotic Kunz's theorem}
There is an asymptotic aspect behind the local Kunz's theorem, namely, the $F$-signature. Let $R$ be a complete local algebra. We may define $0 \leq a_e \leq q^{\dim R}$ to be the largest rank of a free quotient of $F^e_* R$ as an $R$-module. Then, the limit $0 \leq \lim_{e \rightarrow \infty} a_e/q^{\dim R} \leq 1$ exists, it is called the $F$-signature of $R$, and is denoted by $s(R)$. Then $s(R)=1$ if and only if $R \cong \kay \llbracket x_1, \ldots ,x_{\dim R}\rrbracket$. See \cite{TuckerFSigExists} for details.

\begin{proposition} \label{pro.VolumeSignatureFormula}
Let $X$ be an $F$-split smooth projective variety. Then, for every invertible sheaf $\sL$ on $X$, the following formula for computing its volume (see \cite[Definition 2.2.31]{LazarsfeldPositivity1}) holds:
\[
\vol_X (\sL) = \lim_{e \rightarrow \infty} \frac{h^0(X,\sL \otimes \sE_{e,X}^{\vee})}{q^d/d!} \eqqcolon \epsilon(\sL).
\]
\end{proposition}
\begin{proof}
Twist the split sequence \autoref{eqn.FrobeniusSES} by $\sL$ and use the projection formula to conclude that $h^0(X,\sL^q)=h^0(X,\sL)+h^0(X,\sL \otimes \sE_{e,X}^{\vee})$. Dividing by $q^d/d!$ and letting $e \rightarrow \infty$ yield the desired equality.
\end{proof}
In this way, with notation as in \autoref{pro.VolumeSignatureFormula}, if $X$ admits a very ample invertible sheaf $\sL$ such that $\epsilon(\sL)=1$ then $X \cong \bP^{\dim X}$ (as for a very ample $\sL$ its volume equates to the degree of the closed embedding $i \: X \to \bP\bigl(H^0(X,\sL)\bigr)$). For instance, if $X$ admits a decomposition
\[
F^e_* \sO_X \cong \sO_X \oplus (\sL^{-1})^{a_{e}} \oplus \sF_e
\]
such that $h^0(\sL \otimes \sF_e) = 0$ and $\lim_{e \rightarrow \infty} a_e/(q^d/d!) = 1$ (as the projective spaces do) then $\epsilon(\sL)=1$ and $X \cong \bP^{\dim X}$.

\subsubsection{Miscellaneous} We may wonder about the structure of the mapping $\sL \mapsto \det F^e_* \sL$ on $\Pic X$. We may further consider the mapping $\alpha\:\Pic X \to \Pic X$ given by
\[
\alpha\: \sL \mapsto \bigotimes_{n=0}^{q-1}\det F^e_* \sL^n.
\]
We can compute this for $X = \bP^d$ and $\sL= \sO(1)$. Let $\alpha\bigl(\sO(1)\bigr) \cong \sO(-a)$. We compute $a$ as follows. If $f(t) = \sum_{l \geq 0} a_{l} t^ l= \sum_{n=0}^{q-1} \sum_{i \geq 0} a_{i,n}t^{iq+n}$, then 
\[f'(t) = \sum_{n=0}^{q-1} \sum_{i \geq 0} (iq+n)a_{i,n}t^{iq+n-1} = q \sum_{n=0}^{q-1} \sum_{i \geq 0} i a_{i,n}t^{iq+n-1} + \sum_{n=0}^{q-1}n\sum_{i \geq 0}a_{i,n}t^{iq+n-1}.
\] 
Therefore, setting $t=1$, we have 
\[f'(1) = q \sum_{n=0}^{q-1} \sum_{i \geq 0} ia_{i,n} + \sum_{n=0}^{q-1}n\sum_{i \geq 0} a_{i,n}\]
Therefore, applying this to $f(t)=\big((1-t^q)/(1-t)\big)^{d+1}$ gives
\[
(d+1)q^d \frac{q(q-1)}{2} = q \cdot a + \sum_{n=0}^{q-1} q^d = q \cdot a + q^d \frac{q(q-1)}{2}
\]
by our calculations in \autoref{pro.PushforwardProjectiveBundles}, where we use that $\sum_{i\geq 0} a_{i,e;d,n} = q^d$. Consequently,
\[
a = \frac{dq^d(q-1)}{2}.
\]
In general, $(F^e_* \sL)^{\vee} \cong F^e_* (\sL^{-1} \otimes \omega_X^{1-q})$. Applying this to $\sL = \omega_X^{-1}$ gives 
\[
F^e_* \omega_X^{-n} \cong F^e_*\big( \omega_X^{q-1-n} \otimes \omega_X^{1-q} \big)= \big(F^e_* \omega_X^{n-(q-1)} \big)^{\vee}
\]
In particular, 
\[
\det F^e_* \omega_X^{-n} \cong \Big(\det F^e_* \omega_X^{n-(q-1)} \Big)^{-1}.
\]
Consequently,
\[
\alpha(\omega_X^{-1}) = \alpha(\omega_X^{-1})^{-1},
\]and so $\alpha(\omega_X^{-1}) = \sO_X$ if there is no $2$-torsion in $\Pic X$. Further, if $p \neq 2$ then
\[\alpha(\omega_X^{-1}) = \det F^e_* \omega_X^{(1-q)/2} = \alpha(\omega_X^{-1})^{-1}.
\]
If $p=2$,\[\alpha(\omega_X^{-1}) = \sO_X = \alpha(\omega_X^{-1})^{-1}\]It is worth observing that, if $p \neq 2$ then $F^e_* \omega_X^{(1-q)/2}$ is self-dual.
\begin{question}
Assume $p\neq 2$. Does the self-dual locally free sheaf $F^e_* \omega_X^{(1-q)/2}$ play any role in telling the projective spaces apart among projective varieties? Consider the following $q$ rank-$q$ locally free sheaves
\[F^e_* \sO_X, F^e_* \omega_X^{-1}, F^e_* \omega_X^{-2}, \ldots, F^e_* \omega_X^{(1-q)/2}, \ldots, F^e_* \omega_X^{3-q}, F^e_* \omega_X^{2-q}, F^e_* \omega_X^{1-q}
\] 
where the opposite sheaves in the list are dual pairs. On $X=\bP^d$, $F^e_* \sO_X$ is the most negative while $F^e_* \omega_X^{1-q}$ is the most positive, and $F^e_* \omega_X^{(1-q)/2}$ sits in between being equally positive and negative. In fact, it is the one in that list with the largest number of copies of $\sO_X$ as a direct summand. In fact, if $a_e$ denotes such number of copies, then $F^e_* \omega_X^{(1-q)/2}$ is the only one for which $\lim_{e \rightarrow \infty}a_e/(q^d/d!) >0$. In this paper, we have only considered $F^e_* \sO_X$ and $F^e_* \omega_X^{1-q}$.
\end{question}

\bibliographystyle{skalpha}
\bibliography{MainBib}

\end{document}